\documentclass[11pt]{amsart}

\usepackage[T1]{fontenc}
\usepackage[utf8]{inputenc}
\usepackage{lmodern}
\usepackage{microtype}
\usepackage{mathtools}
\usepackage{comment}
\usepackage{amssymb}
\usepackage[hmargin=1.08in,vmargin=1in]{geometry}
\usepackage{xcolor}
\usepackage{appendix}
\usepackage[normalem]{ulem}
\colorlet{linkblue}{blue!40!black}
\colorlet{urlblue}{blue!50!black}
\usepackage[
  colorlinks=true,
  linkcolor=red,
  citecolor=blue,
  urlcolor=blue
]{hyperref}
\usepackage{aliascnt}
\usepackage[noabbrev]{cleveref}

\allowdisplaybreaks
\numberwithin{equation}{section}
\newtheorem{theorem}{Theorem}[section]
\newaliascnt{proposition}{theorem}
\newtheorem{proposition}[proposition]{Proposition}
\aliascntresetthe{proposition}
\newaliascnt{lemma}{theorem}
\newtheorem{lemma}[lemma]{Lemma}
\aliascntresetthe{lemma}
\newaliascnt{corollary}{theorem}
\newtheorem{corollary}[corollary]{Corollary}
\aliascntresetthe{corollary}
\theoremstyle{definition}
\newaliascnt{definition}{theorem}
\newtheorem{definition}[definition]{Definition}
\aliascntresetthe{definition}
\theoremstyle{remark}
\newaliascnt{remark}{theorem}
\newtheorem{remark}[remark]{Remark}
\aliascntresetthe{remark}
\newaliascnt{notation}{theorem}
\newtheorem{notation}[notation]{Notation}
\aliascntresetthe{notation}

\DeclareMathOperator{\Ad}{Ad}
\DeclareMathOperator{\Id}{Id}
\DeclareMathOperator{\Spec}{Spec}

\crefname{theorem}{Theorem}{theorems}
\crefname{proposition}{Proposition}{propositions}
\crefname{lemma}{Lemma}{lemmas}
\crefname{corollary}{Corollary}{corollaries}
\crefname{definition}{Definition}{definitions}
\crefname{equation}{equation}{equations}

\title{$q$-oper structures 
on a formal punctured disc}
\author{Rudrendra Kashyap}
\address{Department of Mathematics, University of Pittsburgh,
Pittsburgh, Pennsylvania, USA
\newline
\href{mailto: ruk26@pitt.edu}{ruk26@pitt.edu}}
\author{Rahul Singh}
\address{Yau Mathematical Sciences Center, Tsinghua University,
Beijing, China
\newline
\href{mailto: 95rahul32@gmail.com}{95rahul32@gmail.com}
}

\hypersetup{
 pdftitle={q-Oper Structures on a Formal Punctured Disc},
  pdfauthor={Rudrendra Kashyap,Rahul Singh},
  pdfsubject={Local existence of q-oper structures on the punctured formal disc},
  pdfkeywords={q-connection, q-oper, twisted conjugacy, loop group, Coxeter cell, extended affine Weyl group}
}

\begin{document}

\begin{abstract}
Let $G$ be a connected reductive complex algebraic group and let
$q\in\mathbb{C}^{\times}$ be not a root of unity. We prove that every $(G,q)$-connection on a formal punctured disc admits a $(G,q)$-oper structure.
\end{abstract}

\maketitle
\setcounter{tocdepth}{1}
\tableofcontents
\section{Introduction}
Opers are connections equipped with a reduction to a Borel subgroup satisfying a certain transversality condition. A fundamental existence theorem of Frenkel and Zhu \cite{FZ} states that every flat $G$-bundle on a formal punctured disc admits an oper structure. Equivalently, the forgetful map from $G$-opers to $G$-connections on a formal punctured disc is surjective. For $G=GL_n$, this is the familiar cyclic vector theorem \cite[Lemme~1.3]{Del} for differential modules. For an arbitrary connected reductive group, \cite{FZ} proves the existence of oper structures by introducing deformed affine Springer fibers and showing that a certain fiber contains a so-called regular point \cite[Theorem~4]{FZ}. The parahoric version of the existence of a regular point has also been established in \cite[Theorem 5.13]{HHSZ}. Moreover, \cite[Corollary 5.14]{HHSZ} proved a parahoric analog of the Borel reduction step in \cite{FZ}, however, it does not establish the existence of oper structures.

The corresponding global analog of this problem for smooth projective curves is more subtle. Arinkin \cite{Ar} proved that every meromorphic $G$-connection on a
smooth projective curve admits a possibly degenerate oper structure; in particular, every irreducible meromorphic $G$-connection admits a generic oper structure. Here, using the notations in \cite{Ar}, possibly degenerate oper structure on a flat $G$-bundle is a $B$-reduction such that for any trivialization compatible with the reduction, the connection is a section of $\mathfrak{g}^{(-1)}\otimes K$, and a generic oper means that each negative simple root component is non-zero at the generic point of the curve. The purpose of this paper is to establish the existence theorem over a formal punctured disc in the setting of $q$-difference connections.

Opers provide a geometric description of the spectrum of certain quantum integrable systems. For the Gaudin models associated to a simple Lie algebra, the spectrum of these models are encoded by certain opers on $\mathbb{P}^1$ \cite{FFR,FFT,Fre}: solution of the Bethe ansatz equations are described by non-degenerate Miura opers on $\mathbb{P}^1$ \cite{Fre}. Here, a Miura oper is an oper with additional reduction to a Borel subgroup that is preserved by the underlying connection. We remark that the non-degeneracy condition in \cite{Fre} is different from that in \cite{Ar}. See \cite{Fre} and the references therein for more details. 

In the $q$-difference setting, the $q$-opers are related to the Heisenberg spin chains. More precisely, nondegenerate twisted Miura--Pl\"ucker $q$-opers on $\mathbb{P}^1$ correspond to nondegenerate solutions of $QQ$-systems and Bethe Ansatz equations of XXZ-type spin chains associated with the quantum affine algebra $U_{q}(\hat{\mathfrak{g}})$, where $\mathfrak{g}$ is a simply laced Lie algebra \cite{FKSZ}. Moreover, $q$-opers on a punctured disc are related to $q$-deformed classical $\mathcal W$-algebras through a $q$-difference analog of the Drinfeld--Sokolov reduction. For more details, see \cite{FRS,SS}.

Let $q\in\mathbb{C}^{\times}$ be not a root of unity. Let $\mathcal{O}=\mathbb{C}[[t]],
  F=\mathbb{C}((t))$ and consider
\begin{align}
\sigma_q:F\rightarrow F,\qquad
  \sigma_q(f(t))=f(qt).
\end{align}
Let $G$ be a connected reductive complex algebraic group of semisimple rank $r$.  Fix opposite Borel subgroups
$B_+$ and $B_-$ with a common maximal torus $T$, and put $W_G=N_G(T)/T$.  Let
$\Delta=\{\alpha_i\}_{i=1}^{r}$ be the set of simple roots of the derived root
system determined by $B_+$ and $T$. Let us choose once and for all an order
$\alpha_{1},\ldots,\alpha_{r}$ and put
$c=s_{1}\cdots s_{r}\in W_G$.
For each $i$, choose a lift $\dot s_i\in N_G(T)$ of $s_i$, and write
\[
  \dot c=\dot s_1\cdots\dot s_r\in N_G(T),
\]
which is a lift of $c$.

The notion of a $(G,q)$-connection was defined in \cite[Section~2]{FKSZ} for a simple, simply connected complex algebraic
group $G$ and a principal $G$-bundle over $\mathbb P^1$. The same
definition gives the notion of a $(G,q)$-connection on a principal
$G$-bundle over a formal punctured disc for any connected reductive group $G$. Since every principal $G$-bundle is trivial by Steinberg's theorem \cite[Chapter III, \S2.3]{Ser}, a $(G,q)$-connection over $D^{\times}=\Spec(\mathbb{C}((t)))$ can be represented by an element $A\in G(F)$. Changing the trivialization of the principal $G$-bundle by $g\in G(F)$ gives the $q$-conjugate of $A$ by $g$:
\begin{equation}\label{eq:gauge-intro}
  A\longmapsto A^g:=\sigma_q(g)Ag^{-1}.
\end{equation}
In particular, $(A^g)^h=A^{hg}$.

Since Steinberg's theorem \cite[Chapter III, \S2.3]{Ser} holds for any connected linear algebraic group, a $B_-$-reduction over
$D^{\times}$ can be trivialized by a $q$-gauge
transformation. Then with respect to this trivialization, the $(G,q)$-oper condition (cf. \cite[Definition~2.1]{FKSZ}) reads as follows:
\begin{align}\label{eq:opercondition}
  A\in B_-(F)\dot cB_-(F).
\end{align}
So, the $q$-conjugacy class of $A$ admits a $(G,q)$-oper structure if
and only if it meets this Coxeter cell.  
We note that the $(G,q)$-oper condition \eqref{eq:opercondition} is independent of the choice of $B_{-}$-trivialization. 

Our main result is the following.
\begin{theorem}\label{thm:main-intro}
Let $G$ be a connected reductive complex algebraic group and let
$A\in G(F)$. Then there exists $g\in G(F)$ such that
\[
  A^g\in B_-(F)\dot cB_-(F).
\]
In other words, every $(G,q)$-connection over a formal punctured disc admits a $(G,q)$-oper structure.
\end{theorem}
\begin{remark}
\begin{enumerate}
\item A $(GL_n,q)$-connection over $D^{\times}$ represented by $A\in GL_{n}(F)$ is the same as a $q$-difference module \cite[Definition 1.1.5]{DV} $(M,\Phi_q)$ of rank $n$. One can show that $A$ can be $q$-conjugated to a $(GL_n,q)$-oper if and only if there exists a cyclic vector $m\in M$, that is, an element $m$ such that $\{m,\Phi_{q}(m),\cdots,\Phi_{q}^{n-1}(m)\}$ is an $F$-basis of $M$. The existence of such an element is shown in \cite[Lemma 1.3.1]{DV}. 

\item Let $G$ be a semisimple and simply connected algebraic group. Steinberg proved \cite[Theorem 1.4]{St} that every regular conjugacy class in $G(\mathbb{C})$ intersects the Coxeter cell $B_{-}(\mathbb{C})\dot{c}B_{-}(\mathbb{C})$.
Thus, \cref{thm:main-intro} may be viewed as a $q$-difference analog of Steinberg's theorem for formal $(G,q)$-connections. One of the important differences is that \cref{thm:main-intro} does not require a regularity assumption on the $(G,q)$-connection.
\end{enumerate}
\end{remark}
We are also interested in two important subclasses of $q$-connections: integral and indecomposable $q$-connections. Following \cite{BG}, we make the following definition.
\begin{definition}
    We call a $q$-conjugacy class \emph{integral} if it meets $G(\mathcal O)$.
\end{definition}
In Section \ref{app:altproof} we give a separate proof of \cref{thm:main-intro} in the case of integral $q$-connections.

\begin{theorem}\label{thm:integral-intro}
Let $A\in G(\mathcal O)$. Then there exists $g\in G(\mathcal O)$ such that
\[
  A^g\in B_-(F)\dot cB_-(F).
\]
\end{theorem}
Thus, in the integral case, not only does an oper representative exist, but the required gauge transformation can itself be chosen integral as well.
\begin{definition}
A $q$-conjugacy class in $G(F)$ is called \emph{indecomposable} if
it does not have a
representative in $M(F)$ for any proper Levi subgroup
$M\subsetneq G$.  
\end{definition}
The arguments in \cite{NZ} yield the following useful integrality statement in the semisimple simply connected case. 
\begin{theorem}\label{thm:indecomp-intro}
Let $G$ be a semisimple and simply connected complex algebraic group. Then every indecomposable $q$-conjugacy class in $G(F)$ is integral. 
\end{theorem}
In view of Theorem~\ref{thm:integral-intro}, we obtain that every indecomposable $q$-conjugacy class admits a $(G,q)$-oper structure. 
\subsection{Relation with earlier work} Theorem \ref{thm:main-intro} addresses an existence problem rather than a classification problem. For $|q|<1$ and a connected semisimple complex algebraic group $G$, \cite[Theorem~1.2]{BG} classified the integral $q$-conjugacy classes in $G(F)$ in terms of semistable principal $G$-bundles over the elliptic curve $\mathbb{C}^{\times}/q^{\mathbb{Z}}$. The classification does not assert the existence of $q$-oper structures. Therefore, the two results are complementary. In the integral case, our proof is inspired by the proofs of \cite[Proposition~2.1, Lemma~2.2]{BG}. 

\cite{FRS,SS} constructed a cross-section for the $q$-gauge action of the loop group of the unipotent radical of a Borel subgroup on a Coxeter-type cell (\cite{FRS} in the $\mathfrak{sl}_n$-case and \cite{SS} in the general semisimple simply connected case other than $E_6$). This construction provides coordinates on the space of $(G,q)$-opers on the punctured disc. However, it does not prove that an arbitrary $q$-connection admits a $q$-oper structure. 

Nie and Zhou \cite[Theorem~1.1]{NZ} proved the existence of a certain normal form for arbitrary $q$-connections that expresses each $q$-conjugacy class by a distinguished pair $(s,h)$, where $s$ is straight and $h$ is pro-unipotent and $sh=hs$. This normal form is the starting point of our proof, but it does not produce a $q$-oper representative by itself. The new steps are to pass from a distinguished pair to a partial Coxeter cell and then to find a representative in the Coxeter cell starting from a representative in a partial Coxeter cell. A key role in the proof is played by straight and superbasic elements, which were studied in \cite{HN}.

Although Theorem~\ref{thm:main-intro} is the direct $q$-analog of \cite{FZ}, its proof is not a $q$-deformation of the argument in \cite{FZ}. In the differential setting, if the connection has an expansion
\[
\nabla=d+B(t)\,dt, 
\qquad
B(t)=B_{r}t^{r}+B_{r+1}t^{r+1}+\ldots, \qquad
B_r\neq 0,
\]
then the leading coefficient $B_r\in\mathfrak{g}$. However, for a $q$-connection $A\in G(F)$, the lowest Laurent coefficient in a faithful representation $G\hookrightarrow GL_n$ need not belong to $G$ or $\mathfrak{g}$. For example, 
\[
A=\begin{pmatrix}
    t & 0\\
    0 & t^{-1}
\end{pmatrix}
=
t^{-1}\begin{pmatrix}
    0 & 0\\
    0 & 1
\end{pmatrix}
+
t\begin{pmatrix}
    1 & 0\\
    0 & 0
\end{pmatrix}
\in
SL_{2}(F).
\]
However, if $A\in G(\mathcal{O})$ is an integral $q$-connection and $g\in G(\mathcal{O})$, then the constant term behaves well under the corresponding $q$-gauge transformation:
\[
A^g(0)=g(0)A(0)g(0)^{-1},
\]
where for $h\in G(\mathcal{O})$, $h(0)$ denotes the reduction of $h$ mod $t$. This is one of the reasons why the proof of Theorem~\ref{thm:main-intro} is shorter for integral $q$-connections.

\subsection{Ideas of the proofs}
The proof of Theorem~\ref{thm:main-intro} consists of two steps: finding a representative in a partial Coxeter cell and then passing to the full Coxeter cell.

For the first step, Lemma~\ref{lem:NZ-newton-levi} gives a representative of the $G(F)$-conjugacy class of $A^{-1}\sigma_q$ (see Subsection \ref{conventions} for this notation) of the form $sh=hs$ in a Levi subgroup, where $s=n\sigma_q$ is standard with its projection to the extended affine Weyl group to be of length-zero and $h$ is pro-unipotent. Viewing $h$ as a unipotent element of the centralizer
of $s$, Proposition~\ref{prop:parabolic} uses an associated cocharacter to place $h$ in the unipotent radical of a parabolic subgroup while keeping $n$ in its Levi subgroup.
The Lemma~\ref{lem:inherit-standard} ensures that $s$ remains standard in this Levi subgroup. The Corollary~\ref{prop:superbasic-reduction} then reduces
$s$ further to a superbasic element in a minimal Levi subgroup. As a consequence, the Dynkin diagram of the derived subgroup of this Levi subgroup has only type $A$ components, and the finite Weyl component of the standard factor is conjugate to a product of Coxeter elements (see Lemma~\ref{lem:elliptic-superbasic}). After choosing a compatible Borel subgroup, we get
Theorem~\ref{thm:partial-cell}: the $q$-conjugacy class of $A$ contains a representative
\[
  C\in B_-(F)\dot c_JB_-(F)
\]
for some $J\subset\Delta$.

The second step in the proof of Theorem~\ref{thm:main-intro} uses the $q$-twisted version of the Lang map
\[
  \mathcal L_C(k)=C^{-1}\sigma_q(k)Ck^{-1}.
\]
It can be shown that the linearization of $ \mathcal{L}_{C}$ at $1$ is
$\operatorname{Ad}(C^{-1})\circ\sigma_q-\operatorname{Id}$. The Lemma~\ref{lem:right-inverse-lattice} constructs a continuous right inverse to the operator $\operatorname{Ad}(C^{-1})\circ\sigma_q-\operatorname{Id}$ on sufficiently small $t$-adic neighborhood of $0\in\mathfrak{g}(F)$, with a uniform bound on the loss of valuation. Using the exponential map and the fixed point theorem, we show that $\mathcal{L}_C$ is quasi-open (Theorem~\ref{thm:lang-open}). Consequently, every sufficiently small multiplicative deformation of $C$ can be realized by a $q$-gauge transformation of $C$.

Next,
we construct a sequence $v_n\to1$ in Lemma~\ref{lem:coxeter-approximation} such that
\[
  Cv_n\in B_-(F)\dot cB_-(F).
\]
This construction inserts factors $x_i(t^n)$ in the positions of the simple reflections omitted from $c_J$.
By quasi-openness, we can write $v_n=\mathcal L_C(k)$ for sufficiently large $n$, and hence $C^k=Cv_n$ belongs to the full Coxeter cell. This proves
Theorem~\ref{thm:promotion-intro}, which, together with
Theorem~\ref{thm:partial-cell}, proves
Theorem~\ref{thm:main-intro}.

\smallskip
For the proof of Theorem~\ref{thm:integral-intro}, we give a separate argument that avoids the reduction to superbasic elements. Let $A\in G(\mathcal O)$ and let $a_0\in G(\mathbb{C})$ denote the reduction of $A$ modulo $t$. Let $a_0=s_0u_0$ be the Jordan decomposition of $a_0$. On the successive quotients of the higher congruence subgroups described in Proposition~\ref{lem:congruence}, the change in degree $m$ under a $q$-gauge transformation is governed by the following operator on $\mathfrak{g}$ (see the proof of Proposition~\ref{prop:aligned}):
\[
  L_m(a_0)=q^m\operatorname{Ad}(a_0^{-1})
           -\operatorname{Id}_{\mathfrak g}.
\]
The operator $L_m(a_0)$ is
invertible away from the resonant eigenspace (see Lemma~\ref{lem:nonresonant})
\[
  E_m=\{X\in\mathfrak g:
          \operatorname{Ad}(s_0)X=q^mX\}.
\]
The Lemma~\ref{lem:compatible-borel} shows the existence of a Borel
subgroup $B$ that contains $a_0$ and 
$E_m\subset\operatorname{Lie}R_u(B)$ for every $m\geq1$.
Eliminating the nonresonant terms gives the
normal form of Proposition~\ref{prop:aligned}:
\[
  A^g=a_0\exp(tY_1)\exp(t^2Y_2)\cdots\exp(t^MY_M),
  \qquad g\in G_1,\quad Y_m\in E_m.
\]
Thus, the resulting representative lies in
$B(\mathcal O)$, and then conjugation by an element of $G(\mathbb{C})$ gives
Corollary~\ref{cor:borel-normal}: an integral $q$-gauge transforms $A$ into an element $C\in B_-(\mathcal O)$.

The final step in the proof of Theorem~\ref{thm:integral-intro} is to insert the Coxeter factors while
preserving the integral $q$-conjugacy class.
Proceeding similarly as in the proof of Proposition~\ref{prop:aligned} gives Theorem~\ref{thm:rigidity}: every $D\in G(\mathcal O)$ congruent to $C$ modulo $t^N$ has the form $C^k$ for some $k\in G_N$, where $N$ is large enough so that $E_m=0$ for $m\geq N$. Now Lemma~\ref{lem:coxeter-factor} gives
\[
  v_N=x_1(t^N)\cdots x_r(t^N)
  \equiv1\pmod{t^N},
  \qquad
  v_N\in B_-(F)\dot cB_-(F).
\]
Since $C\in B_-(\mathcal O)$, the element $D:=Cv_N$ belongs to the full Coxeter cell and is congruent to $C$ modulo $t^N$ for all $N$. As a conseqeunce, Theorem \ref{thm:rigidity} realizes this deformation of $C$ by an integral $q$-gauge transformation due to Lemma \ref{lem:finite-resonance}. Combining it with
Corollary~\ref{cor:borel-normal} proves
Theorem~\ref{thm:integral-intro}.

\subsection*{Organization}
In \cref{sec:preliminaries} we give the basic definitions and recall the Nie--Zhou
normal form. In \cref{sec:type-A} we reduce the standard factor to a superbasic element in a Levi subgroup and prove the existence of a representative in a partial Coxeter cell. In Section~\ref{sec:promotion} we establish quasi-openness of a
$q$-twisted version of the Lang map and prove the existence of a $q$-oper structure on a $q$-connection lying in a partial Coxeter cell. In Section~\ref{sec:rigidity} we provide an alternative proof of the main theorem for integral $q$-connections.
\subsection{Acknowledgements}The authors thank Daniel S. Sage for proposing this problem and for several useful discussions. RK would like to thank Roman Fedorov and Bogdan Ion for helpful conversations about this problem. RK is partially supported by the NSF grant DMS-2402553. RS thanks Yongchang Zhu for helpful discussions. 
\section{Preliminaries}\label{sec:preliminaries}
\subsection{Conventions}\label{conventions}
Nie and Zhou \cite{NZ} use the $q$-conjugation convention:
\begin{equation}\label{eq:nzconvention}
A'\longmapsto x^{-1}A'\sigma_q(x).
\end{equation}
It is related to \eqref{eq:gauge-intro} by inversion: if $A'=A^{-1}$ and $x=g^{-1}$, then
\begin{equation}\label{eq:conventions}
  (A^g)^{-1}=x^{-1}A'\sigma_q(x).
\end{equation}
In other words, the $q$-conjugacy class of $A$ in the convention \eqref{eq:gauge-intro} corresponds under inversion to the $q$-conjugacy class of $A^{-1}$ in the convention \eqref{eq:nzconvention}.

Consider the semidirect product $G(F)\rtimes\mathbb{C}^\times$, where $z\in\mathbb{C}^\times$ acts on $G(F)$ through
the automorphism induced by $t\mapsto zt$. Then, under the identification $A'\mapsto A'\sigma_q$, the $q$-conjugacy class of $A'$ in the convention \eqref{eq:nzconvention} is identified with the ordinary $G(F)$-conjugacy class of $A'\sigma_q$ in $G(F)\rtimes\mathbb{C}^\times$.
\subsection{Affine Weyl groups}

Let $H\subset G$ be a connected reductive subgroup containing $T$.
Denote its finite Weyl group by $W_H$ and its coroot lattice by
$Q_H^\vee$.  The extended affine Weyl group is
\[
  \widetilde W_H=N_H(T)(F)/T(\mathcal O)
  \simeq X_*(T)\rtimes W_H,
\]
and its affine Weyl subgroup $W_H^a$ is $Q_H^\vee\rtimes W_H$. Fix the base alcove in the semisimple apartment and write
$\Omega_H\subset\widetilde W_H$ for its stabilizer. Then $W_H^a$ is a Coxeter group and the length function $\ell$ on $W_H^a$ is extended to $\widetilde W_H$ in the following way \cite[Section~1]{HN}:
\[
    \ell(w\delta)=\ell(w), \qquad w\in W_{H}^a, \delta\in\Omega_H.
\]
In particular, $\Omega_H$ is exactly the subgroup of $\widetilde W_H$ consisting of length zero elements.
Every $W_H^a$-coset in $\widetilde W_H$ 
has a unique length-zero representative, giving the standard identification
\begin{equation}\label{eq:omega-pi1}
  \Omega_H\xrightarrow{\sim}
  X_*(T)/Q_H^\vee=\pi_1^{\mathrm{alg}}(H),
\end{equation}
where $\pi_1^{\mathrm{alg}}(H)$ is the algebraic fundamental group of $H$.

Let
$
  N_{0,H}:=\bigl\langle N_H(T)(\mathbb{C}),\ t^\lambda:
  \lambda\in X_*(T)\bigr\rangle\subset N_H(T)(F).
$
The natural map $N_{0,H}\to\widetilde W_H$ is surjective and has kernel
$T(\mathbb{C})$. For $n\in N_H(T)(F)$, we write
\[
  p_H(n\sigma_q)=nT(\mathcal{O})\in\widetilde W_H.
\]
Fix the Iwahori subgroup determined by $B_-\cap H$, and
let $\mathcal{U}_H$ denote its pro-unipotent radical. Following \cite{NZ}, an element
of $H(F)\sigma_q$ is called \emph{standard} if it is
$H(F)$-conjugate to an element of $N_{0,H}\sigma_q$. An element
$n\sigma_q\in N_{0,H}\sigma_q$ is called \emph{straight}
if $\tau=p_H(n\sigma_q)$ satisfies
$\ell(\tau^m)=m\ell(\tau)$, for every $m\geq1$.

For $x\in\widetilde W_H$, there exists $n\in\mathbb{N}$ such that $x^{n}=t^\lambda$ for some $\lambda\in X_{*}(T)$. The \emph{Newton
point} of $x$ is defined as (\cite[Section~2]{HN}):
\[
\nu_x=\frac{\lambda}{n}.
\]
Note that the definition of the Newton point does not depend on the choice of $n$.
\subsection{Coxeter cells}
We write $B_\pm=TU_\pm$. Let us choose a pinning on the triple $(G,T,B_{+})$, that is, for each simple root $\alpha_i$, choose root
homomorphisms
\[
  x_i:\mathbb{G}_a\longrightarrow U_{\alpha_i},\qquad
  y_i:\mathbb{G}_a\longrightarrow U_{-\alpha_i},
\]
and a compatible homomorphism $\varphi_i:\mathrm{SL}_2\to G$. Thus,
\[
  x_i(z)=\varphi_i\!\begin{pmatrix}1&z\\0&1\end{pmatrix},\qquad
  y_i(z)=\varphi_i\!\begin{pmatrix}1&0\\z&1\end{pmatrix},
\]
\[
  \alpha_i^\vee(z)=
  \varphi_i\!\begin{pmatrix}z&0\\0&z^{-1}\end{pmatrix},\qquad
  \dot s_i=\varphi_i\!\begin{pmatrix}0&1\\-1&0\end{pmatrix}.
\]
Since $s_{1}\cdots s_{r}$ is a reduced expression, we have the following.
\begin{equation}\label{eq:bruhat-product}
  (B_-\dot s_{1}B_-)\cdots(B_-\dot s_{r}B_-)
  =B_-\dot cB_-.
\end{equation}
\begin{lemma}\label{lem:coxeter-factor}
For $N\geq1$, set
$
  v_N=x_{1}(t^N)\cdots x_{r}(t^N).
$
Then $v_N\equiv1\pmod{t^N}$ and
$
  v_N\in B_-(F)\dot cB_-(F).
$
\end{lemma}
\begin{proof}
For $z\in F^\times$, the rank-one identity
\begin{equation}\label{eq:rank-one}
x_i(z)=\alpha_i^\vee(-1)\,y_i(z^{-1})\,\alpha_i^\vee(-z)\,\dot s_i\,y_i(z^{-1})
\end{equation}
follows by applying $\varphi_i$ to
\[
\begin{pmatrix}1&z\\0&1\end{pmatrix}
=
\begin{pmatrix}-1&0\\0&-1\end{pmatrix}
\begin{pmatrix}1&0\\z^{-1}&1\end{pmatrix}
\begin{pmatrix}-z&0\\0&-z^{-1}\end{pmatrix}
\begin{pmatrix}0&1\\-1&0\end{pmatrix}
\begin{pmatrix}1&0\\z^{-1}&1\end{pmatrix}.
\]
Thus $x_i(z)\in B_-(F)\dot s_iB_-(F)$. Clearly, $v_N\equiv1\pmod{t^N}$ and the assertion $v_N\in B_-(F)\dot cB_-(F)$ follows by taking $z=t^N$ and using \eqref{eq:bruhat-product}. 
\end{proof}
\begin{notation}\label{cJ}
For $J\subset\Delta$, we let $c_J$ denote the subword of
$c=s_{1}\cdots s_{r}$ obtained by retaining (without changing the fixed order on the set of simple roots) exactly the factors whose simple roots belong to $J$. It is a Coxeter element of the Weyl group of the standard Levi $M_J$. We denote the corresponding lift of $c_{J}$ by $\dot{c}_J$. For $J=\varnothing$, we set $c_J=1$. We call $B_{-}(F)\dot{c}_JB_{-}(F)$ the partial Coxeter cell associated with $J$.
\end{notation}
\subsection{A normal form}
It is known (\cite[Theorem~1.1]{NZ}) that every element of $G(F)\sigma_q$ has a unique factorization as a product of commuting
standard and pro-unipotent factors, and that every
$G(F)$-conjugacy class in $G(F)\sigma_q$ contains a distinguished representative $sh=hs$, with $s$ straight and $h$ pro-unipotent.
 
The next lemma follows from the proof of \cite[Proposition~3.1]{NZ}.
\begin{lemma}\label{lem:NZ-newton-levi}
Let $H$ be a connected reductive complex algebraic group with a maximal torus $T$, and
let $\mathcal C$ be an $H(F)$-conjugacy class in
$H(F)\sigma_q$. Then there exists a Levi subgroup $M\subset H$ that contains $T$ and a representative $sh=hs\in\mathcal C$ such that
\begin{equation}\label{eq:NZ-data}
  s=n\sigma_q,\qquad n\in N_{0,M},\qquad
  p_M(s)\in\Omega_M,\qquad h\in \mathcal{U}_M,
\end{equation}
where $\mathcal{U}_M$ is the pro-unipotent radical of the Iwahori subgroup of $M(F)$ induced by the chosen alcove. In particular, $s$ is straight in $M(F)\sigma_q$.
\end{lemma}
\begin{proof}
By the proof of \cite[Proposition~3.1]{NZ}, there exist a straight
element $x\in\widetilde W_H$ and an element $s_1\in N_{0,H}\sigma_q,\;
p_H(s_1)=x,$
such that with $  M_x=Z_H(\nu_x),$ one has
\begin{equation}\label{eq:NZ-Levi-intersection}
  \mathcal C\cap s_1\mathcal U_{M_x}\neq\varnothing.
\end{equation}

Let $\bar\nu_x$ be the dominant element in the $W_H$-orbit of
$\nu_x$. Choose $z\in N_H(T)(\mathbb C)$ as in the paragraph
immediately following assertion~(ii) in the proof of
\cite[Proposition~3.1]{NZ}; in particular, the Weyl-group image
of $z$ sends $\nu_x$ to $\bar\nu_x$. Set
\[
  M=Z_H(\bar\nu_x)=zM_xz^{-1},
  \qquad
  s=zs_1z^{-1}.
\]
The cited paragraph gives
\begin{equation}\label{eq:NZ-lengthzero}
p_H(s)\in\Omega_M\subset\widetilde W_M\subset\widetilde W_H,
  \qquad
  \mathcal C\cap s\mathcal{U}_M\neq\varnothing.
\end{equation}
Write $ s_1=n_1\sigma_q, \;  s=n\sigma_q.$ Since $z$ is constant,  $n=zn_1z^{-1}.$
Hence $n\in N_{0,H}$. Moreover,
\[
  p_H(s)=nT(\mathcal O)\in\widetilde W_M.
\]
Choose $n_M\in N_{0,M}$ having the same image as $n$ in
$\widetilde W_H$. Then $nn_M^{-1}\in T(\mathbb C)$. Thus $ n\in N_{0,M}$ and $s$ is straight in $M(F)\sigma_q$.

Under the natural inclusion
$\widetilde W_M\subset\widetilde W_H$, we therefore have $p_M(s)=p_H(s)\in\Omega_M.$

Using \eqref{eq:NZ-lengthzero}, choose
\[
  f=su\in\mathcal C,
  \qquad
  u\in\mathcal U_M.
\]
The final recursive construction in the proof of \cite[Proposition~3.1]{NZ} produces convergent sequences
$g_i,h_i,u_i\in\mathcal U_M$
such that
\[
  h_is=sh_i,
  \qquad
  g_ifg_i^{-1}=h_isu_i,
  \qquad
  u_i\longrightarrow1.
\]
Let $g=\lim_{i\to\infty}g_i,  h=\lim_{i\to\infty}h_i.$
Then $g,h\in\mathcal U_M$ and
\[
  gfg^{-1}=hs=sh.
\]
Since $f\in\mathcal C$, it follows that
\[
  sh=hs\in\mathcal C,
  \qquad
  h\in\mathcal U_M.
\]
\end{proof}
\begin{proof}[Proof of \cref{thm:indecomp-intro}]
Let $A\in G(F)$ be indecomposable and put $C=A^{-1}$.  
Applying \cref{lem:NZ-newton-levi} to the conjugacy class of $C\sigma_q$ gives a Levi subgroup $M$ and a representative $sh=hs$ satisfying \eqref{eq:NZ-data}. Hence, by \eqref{eq:conventions}, the $q$-conjugacy class of $A$ also has a representative in $M(F)$. Since $A$ is indecomposable, it follows that $M=G$.

Now, since $G$ is semisimple and simply connected, $\Omega_G=1$ and thus $p_G(s)=1$. Since the kernel of $N_{0,G}\to\widetilde W_G$ is $T(\mathbb{C})$, one has
$s=a\sigma_q\;\text{for some }a\in T(\mathbb{C})$.
Moreover, $h\in\mathcal U_G\subset G(\mathcal O)$. Thus, $sh\in G(\mathcal{O})\sigma_q$ and hence the conjugacy class of $A$ is integral by \eqref{eq:conventions}. 
\end{proof}
\section{Reduction to Levi subgroups}
\label{sec:type-A}
In this section, we show that every $q$-conjugacy class has a representative in a partial Coxeter cell by reducing the standard factor to a superbasic element in a Levi subgroup.
\subsection{Reduction to parabolic subgroups}

Let $M$ be a connected reductive group with maximal torus $T$, let
\[
  s=n\sigma_q,\qquad n\in N_{0,M},\qquad p_M(s)\in\Omega_M,
\]
and set
\[
  Z_s=M(F)^s
  =\{g\in M(F):sg=gs\}.
\]
In other words, $Z_s$ is the stabilizer of $s$ in $M(F)$ under conjugation.
Under these hypotheses, \cite[Lemma~4.1]{NZ} shows that $Z_s$ is a possibly
disconnected reductive algebraic group over $\mathbb C$.

The following lemma allows the associated cocharacter in Proposition~\ref{prop:parabolic} to be conjugated into $T$.
\begin{lemma}\label{lem:constant-maximal-torus}
The connected reductive group $Z_s^\circ$ has a maximal torus contained in $T$.
\end{lemma}

\begin{proof}
The following proof extracts the relevant arguments from \cite[Lemma~4.1]{NZ}.
Let $\tau=p_M(s)\in\Omega_M$.  Choose $m>0$ such that the finite Weyl
component of $\tau^m$ is trivial.  Then $\tau^m=t^\eta$ for some
$\eta\in X_*(T)$.  Since $\tau^m$ is a translation and $\ell(\tau^m)=0$, we must have 
$\langle\alpha,\eta\rangle=0$ for every root $\alpha$ of $M$. Moreover, since
$n\in N_{0,M}$, we can write
\[
  s^m=t^\eta a\,\sigma_q^m
  \qquad\text{for some }a\in T(\mathbb C).
\]
Put
\[
  H_m=\bigl(M(F)^{s^m}\bigr)^\circ.
\]
According to the proof of \cite[Lemma~4.1]{NZ}, $H_m$ is the
connected reductive group generated by $T(\mathbb C)$ and the finitely many
affine root subgroups on which conjugation by $s^m$ is trivial. Therefore, the fixed positive affine root subgroups together with the maximal torus $T(\mathbb C)$ generate a Borel subgroup
$B_m$ of $H_m$.

Now, conjugation by $s$ induces an algebraic automorphism $\theta$ of $H_m$.
By definition of $H_m$, we have $\theta^m=1$.  Moreover,
$\tau\in\Omega_M$ preserves the set of positive affine roots, so $\theta$
stabilizes the pair $T(\mathbb C)\subset B_m$.  Thus, $\theta$ is
quasi-semisimple. By \cite[Proposition~2.1]{DM}, the fixed-point group
$H_m^\theta$ is reductive and
$\bigl(T(\mathbb C)^\theta\bigr)^\circ$
is a maximal torus of $(H_m^\theta)^\circ=Z_s^\circ$.
\end{proof}
\begin{proposition}\label{prop:parabolic}
Let $h\in Z_s$ be unipotent. Then, after conjugating the pair $(s,h)$ by an
element of $Z_s^\circ$, there exists a parabolic subgroup $P=M_1U\subset M$
with $T\subset M_1$ such that
\[
  n\in N_{M_1}(T)(F),
  \qquad
  h\in U(F).
\]
\end{proposition}

\begin{proof}
If $h=1$, take $M_1=M$ and $U=\{1\}$.  Assume now $h\neq 1$. Put
$e=\log(h)$. Since $h$ commutes with $s$,
the one-parameter subgroup
$\{\exp(ze):z\in\mathbb C\}$ is contained in $Z_s$. It follows that
$h\in Z_s^\circ$. 
Since $\text{char}(\mathbb{C})$ is good for $Z_s^{\circ}$, there exists a cocharacter associated to $e$, that is, $ \lambda:\mathbb G_m\longrightarrow Z_s^\circ$ such that $ \operatorname{Ad}(\lambda(z))e=z^2e$ (see \cite[Section~5]{Jantzen}). The equivariance property of the exponential map then gives 
\[
  \lambda(z)h\lambda(z)^{-1}
  =
  \exp\left(\operatorname{Ad}(\lambda(z))e\right)
  =
  \exp(z^2e)
  \longrightarrow 1
  \qquad \text{as }z\longrightarrow 0.
\]
Thus, by the dynamical description of parabolic subgroups (see \cite[Section~6]{Con}), we have $ h\in R_u\bigl(P_{Z_s^\circ}(\lambda)\bigr)$.

By \cref{lem:constant-maximal-torus} and the conjugacy of maximal tori in
$Z_s^\circ$, we can further conjugate by an element of $Z_s^\circ$ and assume that
$\lambda\in X_*(T)$.  In particular, $s$ is fixed under this conjugation. Since
$\lambda(z)\in Z_s$, we have
\[
  n\sigma_q(\lambda(z))n^{-1}=\lambda(z).
\]
Since $\lambda$ is a cocharacter of $T$, $\sigma_q(\lambda(z))=\lambda(z)$ and therefore 
$n\in Z_M(\lambda)(F)$.  Set 
\[
M_1=Z_M(\lambda),
\qquad
P=P_M(\lambda)=M_1U.
\]
Since $n$ normalizes $T$, it belongs to
$N_{M_1}(T)(F)$.
It is clear that $h\in U(F)$.
\end{proof}
The next lemma shows that the element $n\sigma_q$ obtained in the above proposition is standard in $M_{1}(F)\sigma_q$. 
\begin{lemma}\label{lem:inherit-standard}
Let $M_1\subset M$ be a Levi subgroup containing $T$.  If
\[
  n\in N_{0,M}\cap N_{M_1}(T)(F),
\]
then $n\in N_{0,M_1}$.  In particular, $n\sigma_q$ is standard as an
element of $M_1(F)\sigma_q$.
\end{lemma}

\begin{proof}
The affine Weyl component of $n$ belongs to
\[
  \widetilde W_{M_1}=X_*(T)\rtimes W_{M_1}
  \subset\widetilde W_M.
\]
Choose $n_1\in N_{0,M_1}$ with the same affine Weyl component as $n$.  Then
$nn_1^{-1}$ belongs to the kernel of
$N_{0,M}\to\widetilde W_M$, which is $T(\mathbb{C})$.  Since
$T(\mathbb{C})\subset N_{0,M_1}$, it follows that $n\in N_{0,M_1}$.
\end{proof}

\subsection{Reduction of a straight element}
Our goal in this subsection is to show that every straight element can be $q$-conjugated into a Levi subgroup such that it remains straight in the Levi subgroup and that its projection to the extended affine Weyl group is superbasic.
\begin{lemma}\label{lem:standard-in-levi}
Let $H$ be connected reductive with maximal torus $T$, and let $L\subset H$ be a
Levi subgroup containing $T$.  Suppose that a standard element
$s\in H(F)\sigma_q$ is $H(F)$-conjugate to an element
of $L(F)\sigma_q$.  Then its $H(F)$-conjugacy class
contains an element that is straight in $L(F)\sigma_q$:
\[
  s_L=n_L\sigma_q,
  \qquad n_L\in N_{0,L}.
\]
\end{lemma}

\begin{proof}
Let $a\in L(F)\sigma_q$ be a representative of the given $H(F)$-conjugacy class.  Applying
\cite[Theorem~1.1(b)]{NZ} inside $L$, followed by an $L(F)$-conjugation, gives a distinguished representative
\[
  a'=s_Lh_L=h_Ls_L,
\]
where $s_L\in N_{0,L}\sigma_q$ is straight and $h_L$ is pro-unipotent in
$L(F)$. Next, we claim that $h_L$ is pro-unipotent in $H(F)$. Indeed, let us choose a Borel subgroup $B_H$ of $H$ and the corresponding Iwahori subgroup $\mathcal{I}_H$ of $H(F)$. Then $\mathcal{I}_L:=L(F)\cap\mathcal{I}_H$ is an Iwahori subgroup of $L(F)$.  Since $\mathcal{I}_H$ and $\mathcal{I}_L$ are compatible, the
pro-unipotent radical $\mathcal{U}_L$ of $\mathcal{I}_L$ is the intersection of $L(F)$ with the pro-unipotent radical $\mathcal{U}_H$ of $\mathcal{I}_H$. It follows that $h_L$ is also pro-unipotent when viewed in
$H(F)$.

Since $a'$ is standard in $H(F)\sigma_q$, we have two decompositions of
$a'$ as a product of commuting standard and pro-unipotent elements:
\[
  a'=a'\cdot1
  \qquad\text{and}\qquad
  a'=s_Lh_L.
\]
The uniqueness assertion of \cite[Theorem~1.1(a)]{NZ} now gives
$h_L=1$ and $a'=s_L$.
\end{proof}

For the finite Weyl group $W_H$, let
\[
  V_H=\mathbb R\langle\Phi_H^\vee\rangle
\]
be its reflection representation. Recall \cite{Lus2} that $w\in W_H$ is called
\emph{elliptic} if the fixed point subspace $V_H^w$ of $V_H$ under $w$ is trivial, that is, $V_H^w=\{0\}$. Note that the property of being elliptic is invariant under conjugation. There is an equivalent way \cite{Lus1} to define elliptic elements in a Weyl group. We give a proof of this equivalence for completeness.
\begin{lemma}\label{lem:elliptic-parabolic}
An element $w\in W_H$ is elliptic if and only if it is not contained in a
finite Weyl conjugate of a proper standard parabolic subgroup of $W_H$.
\end{lemma}

\begin{proof}
If a conjugate $\widetilde w$ of $w$ lies in $W_J$ for a proper subset $J$ of the simple roots, then $\tilde{w}$ fixes
pointwise the nonzero orthogonal complement of the span of the coroots in
$J$, so it is not elliptic.  Conversely, suppose that $0\neq v\in V_H^w$.
Choose $u\in W_H$ such that $u(v)$ belongs to the closed dominant chamber,
and let
\[
  J=\{\alpha_i:\langle\alpha_i,u(v)\rangle=0\}.
\]
Since $u(v)\neq0$, the subset $J$ is proper.  The stabilizer of $u(v)$ in
$W_H$ is $W_J$, and therefore $uwu^{-1}\in W_J$.
\end{proof}

\begin{lemma}\label{lem:transitive-length-zero}
Let $\widetilde W_R$ be an extended affine Weyl group with irreducible
reduced root system $R$, and let $\Omega_R$ be its length-zero subgroup.
Suppose that $\tau\in\Omega_R$ acts transitively on the vertices of the
closed fundamental alcove $\overline{\mathfrak{a}}$.
\begin{enumerate}
\item 
Then $R$ is of type $A_{m-1}$ for some $m\geq2$. 
\item Moreover, the action of $\tau$ on the extended Dynkin diagram $\tilde{\Gamma}$ is a rotation
by $k$ vertices with $\gcd(k,m)=1$, and the finite Weyl component of $\tau$
is an $m$-cycle.  In particular, it is conjugate in $W(A_{m-1})$ to a Coxeter element.
\end{enumerate}
\end{lemma}

\begin{proof}
Since $\tau$ acts transitively on the vertices of $\overline{\mathfrak{a}}$, it acts transitively on the affine simple reflections and hence any two vertices in $\tilde{\Gamma}$ have the same number of incident edges
(see \cite[Section~3.5]{HN}). It follows that $R$ is of type $A_{m-1}$ for some $m\geq2$.

For $(2)$, since $\tau$ acts transitively on a diagram of type $\widetilde{A}_{m-1}$, it must be a rotation by (say) $k$ vertices where
$\gcd(k,m)=1$. Under the standard identification
$W(A_{m-1})\simeq S_m$, the finite Weyl component of a one-step
rotation is the $m$-cycle $(1\,2\,\cdots\,m)$. 
Thus,
the finite Weyl component of a $k$-step rotation is
$(1\,2\,\cdots\,m)^k$, which is again an $m$-cycle when
$\gcd(k,m)=1$. Since every $m$-cycle is conjugate in $S_m$ to
$(1\,2\,\cdots\,m)=s_1\cdots s_{m-1}$, the final assertion follows.
\end{proof}
Following \cite[Section~3.5]{HN}, $\tau\in\Omega_H$ is called
\emph{superbasic} if every orbit of its action on the affine simple
reflections, or equivalently, on the vertices of the extended Dynkin diagram $\tilde{\Gamma}$ of $H$, is a union of connected components of $\tilde{\Gamma}$. We have the following lemma (cf. \cite[Section~5.9]{GHKR2}).
\begin{lemma}
\label{lem:elliptic-superbasic}
Let
\[
  s_H=n_H\sigma_q,\qquad n_H\in N_{0,H},
  \qquad \tau=p_H(s_H)\in\Omega_H.
\]
Assume that $s_H$ is indecomposable.
Write
\[
  \tau=t^\mu w,
  \qquad \mu\in X_*(T),\quad w\in W_H.
\]
Then $w$ is elliptic and
$\tau$ is superbasic. Moreover, every connected component of the Dynkin diagram of $H^{\mathrm{der}}$ is of type $A$, and $w$ is $W_H$-conjugate to a product of Coxeter elements, one for each connected component.
\end{lemma}
\begin{proof}
Suppose first that $w$ is not elliptic.  By
\cref{lem:elliptic-parabolic}, there exist $u\in W_H$ and a proper standard
Levi subgroup $H_J\subsetneq H$ such that $u^{-1}wu\in W_J$. Hence,
\[
  u^{-1}\tau u
  =t^{u^{-1}\mu}(u^{-1}wu)
  \in X_*(T)\rtimes W_J=\widetilde W_{H_J}\subset\widetilde{W}_H.
\]
Let $\dot u\in N_H(T)(\mathbb C)$ be a lift of $u$. Then
\[
  p_{H}(\dot u^{-1}(n_H\sigma_q)\dot u)
  =u^{-1}\tau u\in\widetilde W_{H_J}\subset\widetilde{W}_H.
\]
Since the preimage of
$\widetilde W_{H_J}$ under $p_H$ is
$N_{H_J}(T)(F)\sigma_q$, we have $\dot u^{-1}(n_H\sigma_q)\dot u\in H_J(F)\sigma_q$, which contradicts the hypothesis. Thus, $w$ is
elliptic.

Let $\Gamma$ denote the Dynkin diagram of $H$ and let $\widetilde{\Gamma}$ denote the corresponding extended Dynkin diagram.
Since $\tau\in\Omega_H$, $\tau$ stabilizes the
fundamental alcove and thus permutes the vertices of $\widetilde{\Gamma}$. Since $W_H$ is the direct product of the
Weyl groups of the irreducible components of $\Gamma$, and the translations by $X_{*}(T)$ preserve the set of affine hyperplanes corresponding to each connected component of $\Gamma$, $\tau$ preserves every component of $\widetilde{\Gamma}$. Thus, our problem reduces to the case of irreducible Dynkin diagrams.

We now need the following elementary lemma.
\begin{lemma}\label{simplex}
Let 
\[
    \varphi:V\rightarrow V,
    \quad
    v\mapsto g\cdot v+b, \qquad b,v\in V,
    \quad g\in GL(V),
\]
be an affine map preserving a full-dimensional simplex $\Sigma\subset V$ and permuting its vertices.  If the associated permutation on the set of vertices has $r$ orbits, then $\dim(\Sigma^\varphi)=r-1$. Moreover,
\[
  \dim\ker(g-\operatorname{Id})=r-1.
\]
\end{lemma}
\begin{proof}
Let $v_0,\ldots,v_d$ denote the vertices of $\Sigma$ and let us write points of $\Sigma$ in barycentric coordinates:
\[
  x=\sum_{j=0}^{d}a_j(x)v_j,
  \qquad a_j(x)\geq0,
  \qquad \sum_ja_j(x)=1,
  \qquad 
  x\in\Sigma.
\]
Then 
\[
x\in\Sigma^\varphi
\iff
a_j(x)=a_{j'}(x)
\quad\text{whenever $v_j$ and $v_{j'}$ belong to the same
$\varphi$-orbit.}
\]
Thus, we have
$\dim(\Sigma^\varphi)=r-1$.

Let $x_0$ be the barycenter of $\Sigma$. Since $\varphi$ permutes the vertices of $\Sigma$, one has $\varphi(x_0)=x_0$. The fixed-point set of $\varphi$ in $V$ is therefore the affine subspace $x_0+\ker(g-\operatorname{Id}).$ Consequently,
\[
  \Sigma^\varphi
  =
  \Sigma\cap\bigl(x_0+\ker(g-\operatorname{Id})\bigr).
\]
Since $\Sigma$ is full-dimensional and $x_0$ lies in its interior, this
intersection has dimension $\dim\ker(g-\operatorname{Id})$. Therefore, we have $ \dim\ker(g-\operatorname{Id})=r-1.$
\end{proof}
Fix an irreducible component and let $\overline{\mathfrak{a}}$ be its
closed fundamental alcove. Let $r$ be the number of orbits of the action of $\tau$ on the vertices of $\overline{\mathfrak{a}}$. We claim $r=1$. By \cref{simplex} we have $r-1=\dim\ker(w-\operatorname{Id}).$
Since $w$ is elliptic, $\ker(w-\operatorname{Id})=0$, and hence $r=1$. Thus, $\tau$ is superbasic.

The remaining assertions follow from \cref{lem:transitive-length-zero}.
\end{proof}
\begin{corollary}
\label{prop:superbasic-reduction}
Let $H$ be a connected reductive group with maximal torus $T$, and let
$s=n\sigma_q$ be a standard element of $H(F)\sigma_q$.  Then there
exists a Levi subgroup $L\subset H$ that contains $T$, an element
$x\in H(F)$, and an element $n_L\in N_{0,L}$ such that
\begin{equation}\label{eq:superbasic-conjugation}
  x^{-1}n\sigma_q(x)=n_L
\end{equation}
and the element $p_L(n_L\sigma_q)\in\Omega_L$ is superbasic. In particular, every connected component of the Dynkin diagram of
$L^{\mathrm{der}}$ is of type $A$, and the finite Weyl component of
$p_L(n_L\sigma_q)$ is $W_L$-conjugate to a product of Coxeter elements, one for each connected component.
\end{corollary}

\begin{proof}
Choose a Levi subgroup $L\subset H$ of minimal semisimple rank such that $L(F)\sigma_q$ contains an $H(F)$-conjugate of $s$. Without loss of generality, we may assume that
$T\subset L$. By \cref{lem:standard-in-levi}, there exist
$y\in H(F)$ and a straight element
\[
  s_1=n_1\sigma_q,\qquad n_1\in N_{0,L},
\]
such that $y^{-1}sy=s_1$.

Applying \cref{lem:NZ-newton-levi} to $L$,
we get a Levi subgroup
$L_0\subset L$ and an element $x_0\in L(F)$ such that
\[
  x_0^{-1}s_1x_0=s_0h_0=h_0s_0,
  \qquad
  s_0=n_0\sigma_q,
\]
with
\[
  n_0\in N_{0,L_0},\qquad p_{L_0}(s_0)\in\Omega_{L_0},
  \qquad h_0\in \mathcal{U}_{L_0}.
\]
Since $s_0h_0$ is conjugate to $s_1$, it is standard in
$L(F)\sigma_q$. The uniqueness of the standard--pro-unipotent
decomposition \cite[Theorem~1.1(a)]{NZ} therefore gives $h_0=1$.
Since $L_0$ is also a Levi subgroup of $H$, the minimality of $L$ forces $L_0=L$. The remaining assertions now follow by the minimality of $L$ and from \cref{lem:elliptic-superbasic}.
\end{proof}

\subsection{Partial Coxeter cell}

\begin{theorem}\label{thm:partial-cell}
Let $G$ be connected reductive and let $A\in G(F)$. There exists a
subset $J\subset\Delta$ and $g\in G(F)$ such that
\[
  A^g\in B_-(F)\dot c_JB_-(F).
\]
Moreover, every connected component of the Dynkin subdiagram associated with $J$ is of type $A$.
\end{theorem}
\begin{proof}
Put $C=A^{-1}$. By \cref{lem:NZ-newton-levi}, the
$G(F)$-conjugacy class of $C\sigma_q$ contains a representative
\[
  sh=hs,
  \qquad
  s=n\sigma_q,
\]
in $M(F)\sigma_q$ for a Levi subgroup $M\subset G$ containing $T$, with the data satisfying \eqref{eq:NZ-data}. In addition, by \cite[Section~4]{NZ}, the pro-unipotent elements of the possibly disconnected reductive group $M(F)^s$ are precisely its unipotent elements.
Therefore, $h$ is unipotent in $M(F)^s$. 

By \cref{prop:parabolic}, after replacing
$C\sigma_q$ by a $G(F)$-conjugate, there exists a parabolic subgroup $  P=M_1U\subset M$ such that
\[
  C\sigma_q=(n\sigma_q)h=h(n\sigma_q),
  \qquad
  n\in N_{M_1}(T)(F),
  \qquad
  h\in U(F).
\]
Since $n\in N_{0,M}$, \cref{lem:inherit-standard} gives
$n\in N_{0,M_1}$, and hence $n\sigma_q$ is standard in
$M_1(F)\sigma_q$.
By
\eqref{eq:conventions}, the $q$-conjugacy class of $A$ therefore has the
representative
\begin{equation}\label{eq:separated-A}
  A_0=C^{-1}=(hn)^{-1}=n^{-1}u,
  \qquad
  u=h^{-1}\in U(F).
\end{equation}
Applying \cref{prop:superbasic-reduction} to $M_1$, we obtain a Levi
subgroup $L\subset M_1$ containing $T$, an element
$x\in M_1(F)$, and $n_L\in N_{0,L}$ such that
\begin{equation}\label{eq:nL-conjugation}
  x^{-1}n\sigma_q(x)=n_L,
\end{equation}
where every connected component of the Dynkin diagram of
$L^{\mathrm{der}}$ is of type $A$, and the finite Weyl component
$w_L$ of $p_L(n_L\sigma_q)$ is $W_L$-conjugate to a product of
Coxeter elements.

Next, we claim that we can find a Borel subgroup $B'_{-}$ of $G$ that is compatible with the Levi subgroups $L\subset M_1\subset M\subset G$ and contains $U$. Indeed, choose a Borel subgroup $B_{1,-}\subset M_1$ containing $T$ for which $L$
is standard. Then $B'_{M,-}=B_{1,-}U$ is a Borel subgroup of $M$.  Choose a parabolic subgroup $Q=MU_Q\subset G$ with Levi factor $M$ and set $ B'_-:=B'_{M,-}U_Q.$ Then $B'_-$ is a Borel subgroup of $G$ containing $T$, $ B'_-\cap M=B'_{M,-},$ and $M$ is standard with respect to $B'_-.$   Choose $h_0\in N_G(T)(\mathbb C)$ with $ h_0B'_-h_0^{-1}=B_-,$ and let $\overline h_0\in W_G$ be the image of $h_0$.  Let $B'_+$ be the Borel subgroup opposite to $B'_-$ and containing $T$.  Transport the fixed order of $\Delta$ and the chosen rank-one representatives by conjugation
with $h_0^{-1}$.  This gives the ordered simple system $\Delta'$ determined by $B'_+$, with the corresponding Coxeter element $c'$ and the representative
$\dot c'$, such that $h_0\dot c'h_0^{-1}=\dot c.$

Let $J'\subset\Delta'$ denote the subset of simple roots corresponding to $L$.

Since all Coxeter elements are conjugate in an irreducible finite Weyl group \cite[Section~3.16]{HumphreysCoxeter} and the finite Weyl component $w_L^{-1}$ of $n_L^{-1}$ is
$W_L$-conjugate to a product of Coxeter elements, there exists
$z\in N_L(T)(\mathbb C)$ such that the image of
$z^{-1}n_L^{-1}z$ in $\widetilde W_L$ is $ t^\mu c'_{J'}$
for some $\mu\in X_*(T)$.

Thus, we can write
\[
  z^{-1}n_L^{-1}z=t^\mu\dot c'_{J'}t_0
\]
for some $t_0\in T(\mathcal O)$, which gives
\begin{equation}\label{eq:normalizer-coxeter}
  z^{-1}n_L^{-1}z
  \in T(F)\dot c'_{J'}
  \subset B'_-(F)\dot c'_{J'}B'_-(F).
\end{equation}
After $q$-conjugating \eqref{eq:separated-A} with $(xz)^{-1}$, we get
\begin{align*}
  A_0^{(xz)^{-1}}
  &=\sigma_q((xz)^{-1})n^{-1}u(xz)\\
  &=\bigl(z^{-1}n_L^{-1}z\bigr)
    \bigl(z^{-1}x^{-1}uxz\bigr).
\end{align*}
Since $x,z\in M_1(F)$ and $M_1$ normalizes $U$,
the second factor belongs to $U(F)\subset B'_-(F)$.
Together with \eqref{eq:normalizer-coxeter}, we get a representative of the $q$-conjugacy class of $A$ in
$B'_-(F)\dot c'_{J'}B'_-(F)$.

Finally, let $ J:=\overline h_0(J')\subset\Delta.$
Then the connected components of the Dynkin diagram associated with $J$ are of type $A$ because the same is true for $J'$. Moreover,
\[
  h_0\bigl(B'_-(F)\dot c'_{J'}B'_-(F)\bigr)h_0^{-1}
  =B_-(F)\dot c_JB_-(F).
\]
Thus, the $q$-conjugacy class of $A$ meets $ B_-(F)\dot c_JB_-(F).$
\end{proof}
\begin{remark}
    The fact that the Dynkin subdiagram associated with $J$ has connected components of type $A$ is not used in the proof of Theorem~\ref{thm:main-intro}; only the existence of a representative in a partial Coxeter cell is needed.
\end{remark}
\section{From partial to full Coxeter cell}\label{sec:promotion}
Recall $\dot c_J$ from Notation~\ref{cJ}.
Throughout this section, $A\in G(F)$ is arbitrary. The goal of this section is to prove the following result.
\begin{theorem}\label{thm:promotion-intro}
If the $q$-conjugacy class of $A\in G(F)$ meets
$B_-(F)\dot c_JB_-(F)$ for some $J\subset\Delta$, then it
meets $B_-(F)\dot cB_-(F)$.
\end{theorem}
\subsection{The $q$-twisted Lang map}\label{sec:lang}
Motivated by the operator
$q\Ad_{a_0^{-1}}-\Id$ considered in \cite[Section~2]{BG}, we introduce
the following $q$-twisted version of the Lang map:
\begin{equation}\label{eq:lang-map}
  \mathcal{L}_A:G(F)\longrightarrow G(F),\qquad
  \mathcal{L}_A(k)=A^{-1}\sigma_q(k)Ak^{-1}.
\end{equation}
Thus,
\begin{equation}\label{eq:gauge-lang}
  A^k=A\mathcal{L}_A(k).
\end{equation}
In this subsection, we show that $\mathcal L_A$ is quasi-open near $1$, that is, the image of every open neighborhood of $1$ contains an open neighborhood of $1$.

Recall from \cite{Liu} that a $q$-difference module over $F$ is a finite-dimensional $F$-vector
space $V$ equipped with a bijective $\sigma_q$-semilinear map $\Phi:V\rightarrow V$, that is, 
\begin{align}
    \Phi(f(t)v)=f(qt)\Phi(v), \qquad f(t)\in F, v\in V.
\end{align}
Our first lemma states that $\Phi-\operatorname{Id}$ admits a right inverse satisfying a uniform $t$-adic estimate (cf. \cite[Lemma~5.1.1]{GHKR1}).
\begin{lemma}\label{lem:right-inverse-lattice}
Let $(V,\Phi)$ be a $q$-difference module and let $\Lambda\subset V$ be an
$\mathcal{O}$-lattice.  Then there exist integers $N_0\in\mathbb{Z}$ and $l\in\mathbb{Z}_{\geq0}$ and a continuous
$\mathbb{C}$-linear map
$S:t^{N_0}\Lambda\longrightarrow V$ such that
\begin{equation}\label{eq:right-inverse}
  (\Phi-\operatorname{Id})S=\operatorname{Id}
\end{equation}
and
\begin{equation}
  S(t^m\Lambda)\subset t^{m-l}\Lambda,
  \qquad m\geq N_0.
\end{equation}

\end{lemma}

\begin{proof}
 First, assume that the $q$-difference module $(V,\Phi)$ is indecomposable. By the classification of indecomposable $q$-difference modules \cite[Theorem~2.4]{Liu}, there exist integers $r,s$, a scalar $\lambda\in\mathbb{C}^{\times}$ and an integer $d\geq 1$, with $\gcd(r,s)=1$ and $s\geq1$ such that
\begin{align}
    (V,\Phi)\cong J_{d}\otimes V_{r,s}(\lambda).
\end{align}
More explicitly, there exists a basis $e_0,\ldots,e_{s-1}$ of $V_{r,s}(\lambda)$ such that
\[
  \Phi(e_j)=e_{j+1}\quad 0\leq j<s-1,
  \qquad
  \Phi(e_{s-1})=\lambda t^r e_0.
\]
On the factor $J_d$, $\Phi$ is represented by a constant unipotent Jordan
matrix. 
It is easy to see that $\Phi^s=t^rC\sigma_q^s$ for a constant invertible matrix $C$.

Let us first consider the case where $\Lambda$ is the standard $\mathcal O$-lattice generated by the corresponding basis of the tensor product.
Put 
\[
  R=\Phi^s, \qquad Q_s=\operatorname{Id}+\Phi+\cdots+\Phi^{s-1}.
\]
Then $R-\operatorname{Id}=(\Phi-\operatorname{Id})Q_s$. 

If $r\neq 0$, then either $R$ or $R^{-1}$ increases the valuation. Hence, $(R-\operatorname{Id})^{-1}$ exists, namely
\[
(R-\operatorname{Id})^{-1}=
\begin{cases}
    -\sum_{j\geq0}R^j & r>0\\
    \sum_{j\geq1}R^{-j} & r<0.
\end{cases}
\]
Hence, $S:=Q_s(R-\operatorname{Id})^{-1}$ is a right inverse to $\Phi-\operatorname{Id}$.
Since $\Phi^j(\Lambda)$ is an $\mathcal O$-lattice in $V$ for each $0\leq j\leq s-1$, there exists an integer $l_j\geq0$ such that $ \Phi^j(\Lambda)\subset t^{-l_j}\Lambda.$ Put $l=\max_{0\leq j\leq s-1}l_j.$ Since $\Phi$ is $\sigma_q$-semilinear, we have
\[
  \Phi^j(t^m\Lambda)
  =
  q^{jm}t^m\Phi^j(\Lambda)
  \subset t^{m-l}\Lambda,
  \qquad
  m\in\mathbb Z,
  \qquad 0\leq j\leq s-1.
\]
Thus, $Q_s(t^m\Lambda)\subset t^{m-l}\Lambda$. Combining this with the fact that $(R-\operatorname{Id})^{-1}$ above does not decrease the valuation, it follows
that
\[
  S(t^m\Lambda)\subset t^{m-l}\Lambda.
\]
It remains to consider $r=0$. In this case, $\Phi=C\sigma_q$.
If
$x=\sum_nx_nt^n$ and $y=\sum_ny_nt^n$, then
$(\Phi-\operatorname{Id})x=y$ is equivalent to
\begin{equation}\label{eq:zero-slope}
  (q^nC-\operatorname{Id})x_n=y_n
\end{equation}
for every $n$. Since $q$ is not a root of unity, the matrix $q^nC-\operatorname{Id}$ is non-invertible only for finitely many integers $n$.
Let $N_0$ be such that $q^nC-\Id$ is invertible for all $n\geq N_0$. Then
\eqref{eq:zero-slope} defines a coefficientwise right inverse to $\Phi-\Id$ preserving
valuation.

Now, in the case where $(V,\Phi)$ is not necessarily indecomposable and $\Lambda_0$ is the lattice generated by the tensor product basis as above of each indecomposable summand, we can take direct sums of the right inverses on each summand to obtain a right inverse $S_{0}$ to $\Phi-\Id$ such that $S_0$ is defined on $t^{N_{0}}\Lambda_0$ and $S_0(t^m\Lambda_0)\subset t^{m-l_0}\Lambda_0, m\geq N_0$ for some integers $N_0\in\mathbb{Z}$ and $l_0\in\mathbb{Z}_{\geq0}$.

For an arbitrary lattice $\Lambda$, choose integers $a,b\geq0$ such that
\[
  t^a\Lambda_0\subset\Lambda\subset t^{-b}\Lambda_0.
\] 
Now, if $y\in t^m\Lambda$, then
$y\in t^{m-b}\Lambda_0$ and hence for $m\geq b+N_0$, we have
\[
  S_0(y)\in t^{m-b-l_0}\Lambda_0
  \subset t^{m-(a+b+l_0)}\Lambda.
\]
Thus, the restriction of $S_0$ to $t^{b+N_0}\Lambda$ gives the required map and $a+b+l_0$ is the desired constant.
\end{proof}
It is known that $\exp$ induces a homeomorphism $\exp:t\mathfrak{g}(\mathcal{O})\rightarrow G_{1}$, where $G_1$ is the first congruence subgroup of $G(\mathcal{O})$. In the following lemma, we show that $\exp$ induces a homeomorphism from $t^{n}\mathfrak{g}(\mathcal{O})$ onto an open subgroup of $G(F)$ for every integer $n\geq1$. We will then use this lemma to pass from Lemma~\ref{lem:right-inverse-lattice} to the $q$-twisted Lang map $\mathcal{L}_A$.
\begin{lemma}\label{lem:formal-coordinates}
For every $n\geq 1$, the
exponential and the logarithm induce mutually inverse homeomorphisms
\[
\exp:t^n\mathfrak{g}(\mathcal{O})\xrightarrow{\sim}\mathcal G_n,
  \qquad
  \log:\mathcal G_n\xrightarrow{\sim}t^n\mathfrak{g}(\mathcal{O}),
\]
where $\mathcal G_n$ is an open subgroup of $G(F)$ contained in the
$n$-th congruence subgroup $G_n$ (see Section \ref{sec:rigidity} for the definition of $G_n$). 
\end{lemma}
\begin{proof}
Since $\exp$ induces a homeomorphism $\exp:t\mathfrak{g}(\mathcal{O})\rightarrow G_{1}$ and $t^n\mathfrak g(\mathcal O)$ is open in
$t\mathfrak g(\mathcal O)$, $\mathcal G_n=\exp(t^n\mathfrak{g}(\mathcal{O}))$ is open in $G_1$, and hence in
$G(F)$.
Next, choose a closed embedding $\rho:G\hookrightarrow\operatorname{GL}(V)$ and let $V_{\mathcal O}=V\otimes_{\mathbb C}\mathcal O, E=\operatorname{End}_{\mathcal O}(V_{\mathcal O}).$
We regard $\mathfrak g(\mathcal{O}) \subset E.$
It is clear from this embedding that for $n\geq 1$ and $X\in t^n\mathfrak{g}(\mathcal{O})$, $\exp(X)\equiv1\pmod{t^n},$ and therefore $\mathcal G_n\subset G_n.$ 

We now show that $\mathcal G_n$ is a subgroup.
Since the multiplication in the image of $\exp$ is given by the BCH formula (\cite[Chapter~2, Section~6]{Bou}):
\begin{align}\label{eq:BCH}
  \operatorname{BCH}(X,Y)
  =
  X+Y+\frac12[X,Y]
  +\frac1{12}[X,[X,Y]]
  +\frac1{12}[Y,[Y,X]]
  +\cdots,
\end{align}
we have $\exp(X)\exp(Y)
  =\exp\!\left(\operatorname{BCH}(X,Y)\right)
  \in\mathcal G_n$ for $X,Y\in t^n\mathfrak g(\mathcal O)$. Also, $ \exp(X)^{-1}=\exp(-X)\in\mathcal G_n$.
This completes the proof of the lemma. 
\end{proof}
We can now prove that $\mathcal{L}_A$ is quasi-open near $1$, that is, the image of every open neighborhood of $1$ under the map $\mathcal{L}_{A}$ contains an open neighborhood at $1$. 
\begin{theorem}\label{thm:lang-open}
For every $t$-adic neighborhood $\mathcal W$ of $1$ in $G(F)$, the image
$\mathcal{L}_A(\mathcal W)$ contains a neighborhood of $1$.
\end{theorem}

\begin{proof}
On $\mathfrak{g}(F)$, put
\[
  \Phi_A=\operatorname{Ad}(A^{-1})\circ\sigma_q.
\]
Applying \cref{lem:right-inverse-lattice} to the $q$-difference module $(\mathfrak{g}(F),\Phi_A)$ and the lattice $\mathfrak{g}(\mathcal{O})$,
there exist $N_0\in\mathbb{Z},c\in\mathbb{Z}_{\geq 0}$ and a $\mathbb{C}$-linear map $S:t^{N_{0}}\mathfrak{g}(\mathcal{O})\rightarrow\mathfrak{g}(F)$ such that
\begin{equation}\label{eq:linear-estimate}
  (\Phi_A-\operatorname{Id})S=\operatorname{Id},
  \qquad
  S(t^m\mathfrak{g}(\mathcal{O}))\subset t^{m-c}\mathfrak{g}(\mathcal{O}),
  \quad m\geq N_0.
\end{equation}
We claim that 
\[
  \exp(t^{n+c}\mathfrak{g}(\mathcal{O}))
  \subset\mathcal{L}_A\bigl(\exp(t^n\mathfrak{g}(\mathcal{O}))\bigr)
  \subset\mathcal{L}_A(\mathcal W)
\]
for all sufficiently large $n$, which would prove the theorem. 

Since $\Phi_A$ is continuous, there exists $d\geq0$ such that
\begin{equation}\label{eq:phi-base}
  \Phi_A(t^d\mathfrak{g}(\mathcal{O}))\subset \mathfrak{g}(\mathcal{O}).
\end{equation}
Now $\sigma_q$-semilinearity of $\Phi_A$ gives
\begin{equation}\label{eq:phi-shift}
  \Phi_A(t^m\mathfrak{g}(\mathcal{O}))\subset t^{m-d}\mathfrak{g}(\mathcal{O})
\end{equation}
for all $m\in\mathbb Z$. Since $\mathcal{L}_{A}\circ\exp$ is continuous, $\mathcal{L}_A(\exp X)$ belongs to the domain of $\log$ for every $X\in t^n\mathfrak{g}(\mathcal{O})$ for all sufficiently large
$n$. For such $n$ we define
\[
\Psi_A(X)=\log\mathcal{L}_A(\exp X).
\] 
For every $Y\in t^{n+c}\mathfrak{g}(\mathcal{O})$, we would like to find $X\in t^{n}\mathfrak{g}(\mathcal{O})$ such that $\Psi_A(X)=Y$. 

Since $\sigma_q$ commutes with $\exp$, we have
\[
  \mathcal{L}_A(\exp X)=\exp(\Phi_AX)\exp(-X),
\]
and the BCH formula \eqref{eq:BCH} gives
\begin{equation}\label{eq:psi-expansion}
  \Psi_A(X)=(\Phi_A-\operatorname{Id})X+Q_A(X),
\end{equation}
where every Lie monomial occurring in $Q_A(X)$ has total degree at least two in
$X$ and $\Phi_AX$. It follows from \eqref{eq:phi-shift} that for $n$ large enough
\begin{equation}\label{eq:quadratic}
  Q_A(t^n\mathfrak{g}(\mathcal{O}))\subset t^{2n-2d}\mathfrak{g}(\mathcal{O}).
\end{equation}
For $Y\in\mathfrak g(F)$, define 
\[
  v_{\mathfrak g(\mathcal O)}(Y)
  :=
  \sup\bigl\{
    m\in\mathbb Z:
    Y\in t^m\mathfrak g(\mathcal O)
  \bigr\},
  \qquad
  v_{\mathfrak g(\mathcal O)}(0):=+\infty.
\]
Moreover, for $X,X'\in t^n\mathfrak{g}(\mathcal{O})$,
\begin{equation}\label{eq:lipschitz}
  v_{\mathfrak{g}(\mathcal{O})}\bigl(Q_A(X)-Q_A(X')\bigr)
  \geq v_{\mathfrak{g}(\mathcal{O})}(X-X')+n-2d.
\end{equation}
Indeed, every homogeneous Lie monomial that occurs in $Q_A(X)-Q_A(X')$ has one factor either $X-X'$ or $\Phi_A(X-X')$, while the remaining factors
have valuation at least $n-d$ \eqref{eq:phi-shift}.

Choose $n$ large enough so that $\exp(t^n\mathfrak{g}(\mathcal{O}))\subset\mathcal W$, $n>2d+c$, $ n+c\geq N_0,$ and $2n-2d\geq N_0$ so that $\Psi_A$ is defined on $t^n\mathfrak g(\mathcal O)$.

For $Y\in t^{n+c}\mathfrak{g}(\mathcal{O})$, define
\[
  T_Y(X)=S\bigl(Y-Q_A(X)\bigr),
  \qquad X\in t^n\mathfrak{g}(\mathcal{O}).
\]
Using \eqref{eq:linear-estimate}, \eqref{eq:quadratic} and the assumption that $Y\in t^{n+c}\mathfrak{g}(\mathcal{O})$, the map $T_Y$ sends $t^n\mathfrak{g}(\mathcal{O})$ to itself.
Furthermore, \eqref{eq:linear-estimate} and \eqref{eq:lipschitz} imply
\[
  v_{\mathfrak{g}(\mathcal{O})}\bigl(T_Y(X)-T_Y(X')\bigr)
  \geq v_{\mathfrak{g}(\mathcal{O})}(X-X')+n-2d-c.
\]
Thus, $T_Y$ is a contraction on the complete metric space
$t^n\mathfrak{g}(\mathcal{O})$. Let $X$ be its unique fixed point. Then \eqref{eq:linear-estimate} gives
\[
 (\Phi_A-\operatorname{Id})X=Y-Q_A(X),
\]
and \eqref{eq:psi-expansion} gives $\Psi_A(X)=Y$. This completes the proof of the theorem.
\end{proof}
\begin{remark}
\Cref{thm:lang-open} is closely analogous to the results of Hartl--Viehmann for the Frobenius-twisted conjugacy. More precisely, in the notation of \cite{HV}, for each fixed $b$, \cite[Lemma~10.2]{HV} shows that there exists an integer $c_b\geq0$ such that, for every $d\geq0$ and every $g\in I_{d+c_b}$, there is a $k\in I_d$ satisfying $bg=k^{-1}b\sigma^\ast(k).$ Moreover, \cite[Theorem~10.1]{HV} gives a uniform choice of the constant $c_b$ as $b$ varies in a bounded subset. ~\cref{thm:lang-open} establishes the corresponding local property for the automorphism $\sigma_q$.

\end{remark}

\subsection{Full Coxeter cell}
Let $C$ be a $q$-connection in a partial Coxeter cell. We will apply the quasi-openness of $\mathcal{L}_C$ to the following sequence.
\begin{lemma}\label{lem:coxeter-approximation}
Let $J\subset\Delta$ and let
$C\in B_-(F)\dot c_JB_-(F)$.  Then there exists a sequence
$(v_n)_{n\in\mathbb{N}}$ in $G(F)$ such that
\[
  v_n\longrightarrow 1,
  \qquad
  Cv_n\in B_-(F)\dot cB_-(F).
\]
\end{lemma}

\begin{proof}
For $1\leq k\leq r$, define
\[
  z_k(n)=
  \begin{cases}
    \dot s_{k},&\alpha_{k}\in J,\\
    x_{k}(t^n),&\alpha_{k}\notin J.
  \end{cases}
\]
By \eqref{eq:rank-one}, each $z_k(n)$ belongs to
$B_-(F)\dot s_{k}B_-(F)$.  Since the Coxeter word is reduced,
\[
  \gamma_{J,n}=z_1(n)\cdots z_r(n)
  \in B_-(F)\dot cB_-(F).
\]
As $n\to\infty$, $z_k(n)$ tends to $1$ whenever $\alpha_k\notin J$, so
$\gamma_{J,n}\to\dot c_J$.  Put
$u_n=\dot c_J^{-1}\gamma_{J,n}$, so $u_n\to 1$.

Choose a factorization $C=b_1\dot c_Jb_2$ with
$b_1,b_2\in B_-(F)$, and set
\[
  v_n=b_2^{-1}u_nb_2.
\]
Then $v_n\to 1$ and
\[
  Cv_n=b_1\gamma_{J,n}b_2
  \in B_-(F)\dot cB_-(F).
\]
\end{proof}

\begin{proof}[Proof of \cref{thm:promotion-intro}]
Start with a representative
$C\in B_-(F)\dot c_JB_-(F)$, and let $(v_n)_{n\in\mathbb{N}}$ be the
sequence from \cref{lem:coxeter-approximation}. By
\cref{thm:lang-open}, the image of $\mathcal{L}_C$ contains a neighborhood
of $1$. Since $v_n\rightarrow 1$, there exists
$k\in G(F)$ such that $\mathcal{L}_C(k)=v_n$ for some $n$. Using
\eqref{eq:gauge-lang}, we obtain
\[
  C^k=Cv_n\in B_-(F)\dot cB_-(F).
\]
\end{proof}
\begin{proof}[Proof of \cref{thm:main-intro}]
By \cref{thm:partial-cell}, the $q$-conjugacy class of $A$ meets
$B_-(F)\dot c_JB_-(F)$ for some $J\subset\Delta$. The conclusion now follows from \cref{thm:promotion-intro}.
\end{proof}
\section{Integral \texorpdfstring{$q$}{q}-connections}
\label{sec:rigidity}\label{app:altproof}
In this section, we give a proof of \cref{thm:integral-intro}. We begin with basic facts on congruence subgroups.
\subsection{Congruence subgroups}

For $m\geq1$, let
\[
  G_m=\ker\left(G\left(\mathcal{O}\right)\longrightarrow
  G\left(\mathcal{O}/t^m\mathcal{O}\right)\right).
\]
For $C,D\in G(\mathcal{O})$, the notation
$C\equiv D\pmod{t^m}$ means that they have the same image in
$G(\mathcal{O}/t^m\mathcal{O})$.

The following is standard; we include a proof for completeness.
\begin{proposition}
\label{lem:congruence}
For every $m\geq1$, there is a natural isomorphism of additive groups
\begin{equation}\label{eq:graded}
  G_m/G_{m+1}\simeq
  \mathfrak{g}\otimes_{\mathbb{C}}
  (t^m\mathcal{O}/t^{m+1}\mathcal{O})\simeq\mathfrak{g}.
\end{equation}
Under this identification, a class is written $1+t^mX$ modulo $t^{m+1}$,
and
\[
  \sigma_q(1+t^mX)\equiv1+q^mt^mX\pmod{t^{m+1}}.
\]
Moreover,
\begin{equation}\label{eq:completion}
  G(\mathcal{O})\simeq
  \varprojlim_nG(\mathcal{O}/t^n\mathcal{O}),
  \qquad
  \bigcap_{n\geq1}G_n=\{1\}.
\end{equation}
\end{proposition}

\begin{proof}
Set
\[
  R=\mathcal O/t^{m+1}\mathcal O,\qquad
  \overline R=\mathcal O/t^m\mathcal O,\qquad
  J=t^m\mathcal O/t^{m+1}\mathcal O.
\]
Then $J$ is the kernel of $R\to\overline R$ and $J^2=0$.
Choose a closed embedding $G\hookrightarrow\operatorname{GL}_N$.  Every
matrix in the kernel of
$\operatorname{GL}_N(R)\to\operatorname{GL}_N(\overline R)$ has a unique expression $1+Z$ with entries of $Z$ in $J$.  If $f$ is a defining equation of $G$ in a neighborhood of the identity, then $f(1)=0$ and the equality $J^2=0$ makes the first-order Taylor formula exact: $  f(1+Z)=df_1(Z)$.
It follows that $1+Z\in G(R)$ exactly when $Z\in T_1G\otimes_{\mathbb C}J
  =\mathfrak g\otimes_{\mathbb C}J.$
This gives a canonical bijection
\begin{equation}\label{eq:square-zero-kernel}
  \ker\bigl(G(R)\longrightarrow G(\overline R)\bigr)
  \simeq \mathfrak g\otimes_{\mathbb C}J.
\end{equation}
Since
$\mathcal O=\varprojlim_r\mathcal O/t^r\mathcal O$ and $G$ is affine, we have
\begin{equation}\label{eq:points-commute-with-limit}
  G(\mathcal O)
  =\operatorname{Hom}_{\mathbb C\text{-alg}}(\mathcal{O}(G),\mathcal O)
  \simeq
  \varprojlim_rG(\mathcal O/t^r\mathcal O).
\end{equation}

Moreover, since $G$ is smooth over $\mathbb C$, every reduction map
\[
  G(\mathcal O/t^{r+1}\mathcal O)
  \longrightarrow
  G(\mathcal O/t^r\mathcal O)
\]
is surjective. Therefore, for every $g_{m+1}\in G(R)$, we can choose compatible lifts $g_n\in G(\mathcal O/t^{n}\mathcal O)$ of $g_{m+1}$ for all $n\geq m+1$. By
\eqref{eq:points-commute-with-limit}, we get $g\in G(\mathcal O)$ that reduces to $g_{m+1}$ mod $t^{m+1}$. Hence $ G(\mathcal O)\longrightarrow G(R)$ is surjective.

Clearly, the image of $G_m$ in $G(R)$ equals the kernel in
\eqref{eq:square-zero-kernel}, and the kernel of $G_m\to G(R)$ is precisely
$G_{m+1}$. Hence,
\[
  G_m/G_{m+1}\simeq\mathfrak g\otimes_{\mathbb C}J.
\]
The other assertions are clear. 
\end{proof}
As a corollary of \cref{lem:congruence}, we have the following.
\begin{corollary}
    If $C\equiv D\pmod{t^m}$, then there exists a unique $Y\in\mathfrak{g}$ such that
\begin{equation}\label{eq:error}
  C\equiv D(1+t^mY)\pmod{t^{m+1}}.
\end{equation}
\end{corollary}
\subsection{Reduction to a Borel subgroup} 
\label{sec:borel-normal} 
The goal of this subsection is to show that any element of $G(\mathcal{O})$ can be $q$-conjugated to an element of $B_{-}(\mathcal{O})$. We begin with a remark.
\begin{remark}
     It is a standard result that at the level of $\mathbb{C}$-points any element of $G(\mathbb{C})$ can be $G(\mathbb{C})$-conjugated to an element of $B_{-}(\mathbb{C})$.
    However, for $\mathcal{O}$-points, not every element of $G(\mathcal{O})$ can be $G(F)$-conjugated to $B_{-}(\mathcal{O})$, for example,
    \begin{align}
        A=\begin{pmatrix}
            0 & 1 \\
            -1 & 2+t
        \end{pmatrix}\in SL_2(\mathcal{O}).
    \end{align}
    Indeed, the characteristic polynomial of $A$ has discriminant $t(4+t)$, which is not a square in $F$. In particular, $A$ does not have eigenvalues in $F$ and therefore, $A$ cannnot be lower-triangularized even over $F$.
    
    Nevertheless, we prove in \cref{cor:borel-normal} that any element of $G(\mathcal{O})$ can be $q$-conjugated to an element of $B_{-}(\mathcal{O})$.
\end{remark} 
For $A\in G(\mathcal{O})$, let $a_{0}:=A\pmod t\in G(\mathbb{C})$. Motivated by the arguments in \cite{BG}, we consider the following operator on $\mathfrak{g}$:
\[
L_m(a_0)=q^m\operatorname{Ad}(a_0^{-1}) -\operatorname{Id}_{\mathfrak g}.
\] For brevity, write $L_m:=L_m(a_0)$ throughout this subsection.
Write the Jordan decomposition
$
  a_0=s_0u_0=u_0s_0,
$
with $s_0$ semisimple and $u_0$ unipotent.  Choose a Borel subgroup $B_0$ containing $a_0$ and a maximal torus $T_0\subset B_0$ containing $s_0$.  The Jordan factors $s_0$ and $u_0$ both belong to $B_0$.
Decompose
\[
  \mathfrak{g}=\bigoplus_{\lambda\in\mathbb{C}^\times}\mathfrak{g}_\lambda,
  \qquad
  \mathfrak{g}_\lambda=\{X\in\mathfrak{g}:\operatorname{Ad}(s_0)X=\lambda X\}.
\]
For $m\geq1$, set
\[
  E_m=\mathfrak{g}_{q^m},\qquad
  F_m=\bigoplus_{\lambda\neq q^m}\mathfrak{g}_\lambda.
\]

Both summands are stable under $\operatorname{Ad}(u_0)$ since $s_0$ commutes with $u_0$.

\begin{lemma}\label{lem:nonresonant}
The operator $L_m$ preserves $E_m$ and $F_m$, and its restriction to $F_m$ is invertible.
\end{lemma}

\begin{proof}
On $\mathfrak{g}_\lambda$ one has
$
  \operatorname{Ad}(a_0^{-1})=\lambda^{-1}\operatorname{Ad}(u_0^{-1}).
$
Now, since $\operatorname{Ad}(u_0^{-1})$ is unipotent, the restriction of $L_m$ to
$\mathfrak{g}_\lambda$ is the sum of the nonzero scalar
$q^m\lambda^{-1}-1$ and a nilpotent endomorphism whenever
$\lambda\neq q^m$.  It is therefore invertible on $F_m$.
\end{proof}
\begin{lemma}\label{lem:compatible-borel}
There exists a Borel subgroup $B\subset G$ such that
\[
  a_0\in B,
  \qquad
  E_m\subset\operatorname{Lie}R_u(B),\quad m\geq1.
\]
\end{lemma}

\begin{proof}
Let $L=Z_G(s_0)^\circ$.  Since $u_0$ is unipotent and commutes with $s_0$,
$u_0\in L$. The group $B_L=B_0\cap L$ is a Borel subgroup of $L$
containing $u_0$, and $T_0\subset L$.

Let $\Phi=\Phi(G,T_0)$ and put
\[
  \Psi=\{\alpha\in\Phi:\alpha(s_0)\in q^{\mathbb{Z}}\}.
\]
For $\alpha\in\Psi$, define $d(\alpha)\in\mathbb Z$ by $\alpha(s_0)=q^{d(\alpha)}.$ Since $q\in\mathbb C^\times$ is not a root of unity, this is well
defined and extends to a homomorphism $d:\mathbb Z\Psi\longrightarrow\mathbb Z.$

Let $V_{\Psi}:=\operatorname{span}_{\mathbb R}(\Psi)
  \subset X^*(T_0)_{\mathbb R}.$ 
By choosing an $\mathbb R$-linear complement $W\subset X^*(T_0)_{\mathbb R}$ to $V_\Psi$, we extend $d$ to an $\mathbb{R}$-linear functional
$\widetilde d:X^*(T_0)_{\mathbb R}\to\mathbb R$ by setting $\widetilde d|_{V_{\Psi}}=d_{\mathbb R},\; \widetilde d|_W=0.$

Since the natural pairing $X^*(T_0)_{\mathbb R}\times X_*(T_0)_{\mathbb R}
  \longrightarrow\mathbb R$ is perfect, there exists a unique $v\in X_*(T_0)_{\mathbb R}$ such that
\[
  \langle\chi,v\rangle=\widetilde d(\chi)
  \qquad
  \text{for every }\chi\in X^*(T_0)_{\mathbb R}.
\]
In particular, $ \langle\alpha,v\rangle=d(\alpha)
  \;$ for all $\alpha\in\Psi.$

Choose $v_L\in X_*(T_0)_{\mathbb{R}}$ in the open chamber for the root system of $L$ determined by
$B_L$. Since $\Psi$ is finite, for
sufficiently small $\varepsilon>0$, 
$\langle\alpha,v+\varepsilon v_L\rangle$ has the same sign as $d(\alpha)$ whenever $d(\alpha)\neq0$.
If $d(\alpha)=0$, then $\alpha$ is a root of $L$ and $\langle\alpha,v+\varepsilon v_L\rangle
  =\varepsilon\langle\alpha,v_L\rangle.$
Hence, its sign agrees with the positive system of $L$ determined by
$B_L$. Finally, choose a
sufficiently small vector $\tilde{v}\in X_*(T_0)_{\mathbb{R}}$ such that $v_{0}:=v+\varepsilon v_L+\tilde{v}$ is regular without changing any of the above
prescribed signs.

Let $B$ be the Borel subgroup containing $T_0$ determined by $v_0$. By the above construction, $B\cap L=B_L$, so $u_0\in B$; also
$s_0\in T_0\subset B$.  Thus $a_0=s_0u_0\in B$.  For $m\geq1$, the space
$E_m$ is the sum of the root spaces $\mathfrak{g}_\alpha$ satisfying
$\alpha(s_0)=q^m$. For such an $\alpha$, we have $d(\alpha)=m>0$, so it is a positive root determined by
$B$. Therefore $E_m\subset\operatorname{Lie}R_u(B)$.
\end{proof}

\begin{proposition}\label{prop:aligned}
There exist $g\in G_1$, $M\in\mathbb{N}$, and $Y_m\in E_m$, $1\leq m\leq M$
such that
\begin{equation}\label{eq:aligned}
  A^{g}
  =a_0\exp(tY_1)\exp(t^2Y_2)\cdots\exp(t^MY_M).
\end{equation}
In particular, $A$ is $q$-conjugate to an element of $B(\mathcal{O})$ by an element of $G_1$, where $B$ is as in \cref{lem:compatible-borel}.
\end{proposition}

\begin{proof}
We construct $g$ as the limit of a sequence $(g^{(m)})_{m\geq 1}$, where $g^{(m)}$ is constructed recursively as follows:
\[
  g^{(m-1)}=k_{m-1}\cdots k_1, \qquad k_{i}\in G_i.
\]
Suppose $k_{i}\in G_i$ and $Y_{i}\in E_i$, $1\leq i\leq m-1$, have been chosen in such a way that
\[
  A^{g^{(m-1)}}\equiv C_{m-1}\pmod{t^m},
  \qquad
  C_{m-1}=a_0\prod_{j=1}^{m-1}\exp(t^jY_j).
\]
We set $g^{(0)}=1$ and $C_0=a_0$. By
\eqref{eq:error}, there is a unique $Z_m\in\mathfrak{g}$ such that
\begin{equation}\label{eq:aligned-error}
  A^{g^{(m-1)}}
  \equiv C_{m-1}(1+t^mZ_m)\pmod{t^{m+1}}.
\end{equation}
Write
\[
  Z_m=Z_m^{\mathrm{r}}+Z_m^{\mathrm{nr}},
  \qquad
  Z_m^{\mathrm{r}}\in E_m,
  \quad Z_m^{\mathrm{nr}}\in F_m,
\]
and set $Y_m=Z_m^{\mathrm{r}}$. By \cref{lem:nonresonant}, there exists a
unique $X_m\in F_m$ satisfying
\begin{equation}\label{eq:homological}
  \bigl(q^m\operatorname{Ad}(a_0^{-1})-
  \operatorname{Id}_{\mathfrak g}\bigr)X_m=-Z_m^{\mathrm{nr}}.
\end{equation}
Choose $k_m\in G_m$ whose class modulo $G_{m+1}$ is $1+t^mX_m$. Clearly, we have 
\[
  \sigma_q(k_m)\equiv1+q^mt^mX_m \pmod{t^{m+1}},
  \qquad
  k_m^{-1}\equiv1-t^mX_m \pmod{t^{m+1}}.
\]
Using \eqref{eq:aligned-error} and moving the left factor past
$C_{m-1}$ gives
\begin{align*}
  \sigma_q(k_m)A^{g^{(m-1)}}k_m^{-1}
  &\equiv C_{m-1}\Bigl(1+t^m\bigl[Z_m
       +q^m\operatorname{Ad}(C_{m-1}^{-1})X_m-X_m\bigr]\Bigr)\\
  &\equiv C_{m-1}(1+t^mY_m)
  \pmod{t^{m+1}},
\end{align*}
where we use the fact that $C_{m-1}\equiv a_0\pmod{t}$, which gives 
$t^m\operatorname{Ad}(C_{m-1}^{-1})X_m\equiv
t^m\operatorname{Ad}(a_0^{-1})X_m$ modulo $t^{m+1}$; then
\eqref{eq:homological} cancels the nonresonant part $Z_m^{\mathrm{nr}}$.  Since
$\exp(t^mY_m)\equiv1+t^mY_m\pmod{t^{m+1}}$, the recursion continues.

Since $k_i\in G_i$, the sequence $(g^{(m)})_{m\geq 1}$ is Cauchy and therefore converges to an element $g:=\lim_{m\to\infty}g^{(m)}\in G_1$. Since $q\in\mathbb{C}^{\times}$ is not a root of unity and $\mathfrak{g}$ is finite dimensional,
$E_m=0$ for all sufficiently large $m$, hence the product on the right-hand side of \eqref{eq:aligned}
stabilizes. Passing to the inverse limit yields
\eqref{eq:aligned}.

Finally, since $a_0\in B$ and $Y_m\in\operatorname{Lie}R_u(B)$ by
\cref{lem:compatible-borel},
the right-hand side of \eqref{eq:aligned} lies in
$B(\mathcal{O})$.
\end{proof}

\begin{corollary}\label{cor:borel-normal}
For every $A\in G(\mathcal{O})$, there exists $g\in G(\mathcal{O})$ such that
$A^g\in B_-(\mathcal{O})$.
\end{corollary}

\begin{proof}
Apply \cref{prop:aligned} to get $g'\in G_1$ such that $A^{g'}\in B(\mathcal{O})$, and choose $h\in G(\mathbb{C})$ with
$hBh^{-1}=B_-$.  Then $A^{hg'}\in B_-(\mathcal{O})$.
\end{proof}
\subsection{Proof of the integral theorem} 
For any $A\in G(\mathcal{O})$, put $a_0:=A\bmod t\in G(\mathbb C)$.
\begin{theorem}\label{thm:rigidity}
Let $C,D\in G(\mathcal O)$ be such that
$C\equiv D\pmod{t^N}$ for some $N\geq 1$. Moreover, assume that $L_m(c_0)$
is invertible for all $m\geq N$. Then there exists $k\in G_N$ such
that $C^k=D$.
\end{theorem}

\begin{proof}
We construct $k$ as a limit of a sequence of elements in $G_N$. Suppose that $k^{(m-1)}\in G_N$ has been
chosen so that
\[
  C^{k^{(m-1)}}\equiv D\pmod{t^m}.
\]
At the initial stage $m=N$, take $k^{(N-1)}=1$.  There is a unique
$Y_m\in\mathfrak{g}$ such that
\[
  C^{k^{(m-1)}}\equiv D(1+t^mY_m)\pmod{t^{m+1}}.
\]
Set $X_m=-L_m(c_0)^{-1}Y_m$, and choose $k_m\in G_m$ whose class modulo $G_{m+1}$ is $1+t^mX_m$. We have
\begin{align}
  \sigma_q(k_m)C^{k^{(m-1)}}k_m^{-1}
  \equiv \left(1+q^mt^mX_m\right)D\left(1+t^mY_m\right)\left(1-t^mX_m\right)\\
  \equiv D\left(1+t^m[Y_m+L_m(c_0)X_m]\right)
  \equiv D\pmod{t^{m+1}}.
\end{align}
Thus, with $k^{(m)}=k_mk^{(m-1)}$, the error is moved one degree higher, that is, the image of $D^{-1}C^{k^{(m)}}$ in $G_m/G_{m+1}$ becomes the identity.
The limit $k:=\lim_{m\to\infty}k^{(m)}$ lies in $G_N$. Passing to the limit gives $C^k=D$.
\end{proof}
\begin{remark}
The element $k\in G_N$ in \cref{thm:rigidity} is in fact unique. Indeed, suppose that $k_1,k_2\in G_N$ both satisfy $C^{k_i}=D$, and set
$u=k_2k_1^{-1}$.  Then $D^u=D$.  If $u\neq1$, choose the least $m\geq N$
for which
\[
  u\equiv1+t^mX\pmod{t^{m+1}},\qquad X\neq0.
\]
The degree-$m$ part of $D^u=D$ is $L_m(c_0)X=0$, contradicting the
invertibility of $L_m(c_0)$.
\end{remark}
We now need a lemma that guarantees invertibility of $L_{m}(c_0)$ for all sufficiently large $m$.
\begin{lemma}\label{lem:finite-resonance}
For every $c_0\in G(\mathbb{C})$, the operator $L_m(c_0)$ is invertible for
all sufficiently large $m$.
\end{lemma}

\begin{proof}
Since $\dim(\mathfrak{g})< \infty$, $E_m=0$ for
all sufficiently large $m$. The lemma now follows from Lemma \ref{lem:nonresonant}.
\end{proof}
\begin{proposition}\label{prop:insertion}
Let $C\in B_-(\mathcal{O})$.  There exists $g\in G_1$ such that
\[
  C^g\in B_-(F)\dot cB_-(F).
\]
\end{proposition}

\begin{proof}
Choose $N$ such that $L_m(c_0)$ is invertible for all $m\geq N$, as is
possible by \cref{lem:finite-resonance}.  Let $v_N$ be as in
\cref{lem:coxeter-factor} and put $D=Cv_N$.  Then
$D\equiv C\pmod{t^N}$ and
\[
  D\in B_-(F)\dot cB_-(F).
\]
By \cref{thm:rigidity}, there exists $g\in G_N\subset G_1$ such that $C^g=D$.
\end{proof}
\begin{proof}[Proof of \cref{thm:integral-intro}]
By \cref{cor:borel-normal}, there exists $g_1\in G(\mathcal O)$ such that $ A^{g_1}\in B_-(\mathcal O).$
Applying \cref{prop:insertion}, we obtain $g_2\in G_1$ such that
\[
  (A^{g_1})^{g_2}
  =
  A^{g_2g_1}
  \in B_-(F)\dot cB_-(F).
\]
Since $g_2g_1\in G(\mathcal O)$, the result follows.
\end{proof}
\bibliographystyle{alphanum}
\bibliography{references}

@article{Ar,
  title={Irreducible connections admit generic oper structures},
  author={Arinkin, Dima},
  journal={arXiv preprint arXiv:1602.08989},
  year={2016}
}

@article{BG,
  title={Conjugacy classes in loop groups and {$G$}-bundles on elliptic curves},
  author={Baranovsky, Vladimir and Ginzburg, Victor},
  journal={arXiv preprint alg-geom/9607008},
  year={1996}
}

@book{Bou,
  title={Lie groups and Lie algebras: Chapters 1-3},
  author={Bourbaki, Nicolas},
  volume={1},
  year={1975},
  publisher={Springer Science \& Business Media}
}

@misc{Con,
  author       = {Conrad, Brian},
  title        = {Reductive groups over fields},
  year         = {2020},
  howpublished = {Lecture notes, Stanford University},
  note         = {Notes by Tony Feng. Version of January 19, 2020},
  url          = {https://math.stanford.edu/~conrad/249BW16Page/handouts/249B_2016.pdf}
}

@book{Del,
  title={{\'E}quations diff{\'e}rentielles {\`a} points singuliers r{\'e}guliers},
  author={Deligne, Pierre},
  year={2006},
  publisher={Springer}
}

@article{DM,
  title={Points fixes des automorphismes quasi-semi-simples},
  author={Digne, Fran{\c{c}}ois and Michel, Jean},
  journal={Comptes Rendus Mathematique},
  volume={334},
  number={12},
  pages={1055--1060},
  year={2002},
  publisher={Elsevier}
}

@article{DV,
  title={Arithmetic theory of $q$-difference equations. {T}he $q$-analogue of {G}rothendieck-{K}atz's conjecture on $p$-curvatures},
  author={Di Vizio, Lucia},
  journal={arXiv preprint math/0104178},
  year={2001}
}

@article{FFR,
  title={Gaudin model, {B}ethe ansatz and critical level},
  author={Feigin, Boris and Frenkel, Edward and Reshetikhin, Nikolai},
  journal={Communications in Mathematical Physics},
  volume={166},
  number={1},
  pages={27--62},
  year={1994},
  publisher={Springer}
}

@article{FFT,
  title={Gaudin models with irregular singularities},
  author={Feigin, Boris and Frenkel, Edward and Laredo, V. Toledano},
  journal={Advances in Mathematics},
  volume={223},
  number={3},
  pages={873--948},
  year={2010},
  publisher={Elsevier}
}

@article{FKSZ,
  title={$ q $-opers, {QQ}-systems, and {B}ethe {A}nsatz},
  author={Frenkel, Edward and Koroteev, Peter and Sage, Daniel S. and Zeitlin, Anton M.},
  journal={Journal of the European Mathematical Society},
  volume={26},
  number={1},
  pages={355--405},
  year={2023}
}

@article{FRS,
  author = {Frenkel, E. and Reshetikhin, N. and
            Semenov-Tian-Shansky, M. A.},
  title = {{Drinfeld--Sokolov} reduction for difference operators
           and deformations of {$W$}-algebras.
           {I}. {The} case of {Virasoro} algebra},
  journal = {Communications in Mathematical Physics},
  volume = {192},
  number = {3},
  pages = {605--629},
  year = {1998},
  doi = {10.1007/s002200050311},
  eprint = {q-alg/9704011},
  archivePrefix = {arXiv}
}

@article{FZ,
  title={Any flat bundle on a punctured disc has an oper structure},
  author={Frenkel, Edward and Zhu, Xinwen},
  journal={Mathematical Research Letters},
  volume={17},
  number={1},
  pages={27--37},
  year={2010},
  publisher={International Press of Boston, Inc. Somerville, MA 02143, USA}
}

@incollection{Fre,
  title={Gaudin model and opers},
  author={Frenkel, Edward},
  booktitle={Infinite dimensional algebras and quantum integrable systems},
  pages={1--58},
  year={2005},
  publisher={Springer}
}

@inproceedings{GHKR2,
  title={Dimensions of some affine {D}eligne-{L}usztig varieties},
  author={G{\"o}rtz, Ulrich and Haines, Thomas J. and Kottwitz, Robert E. and Reuman, Daniel C.},
  booktitle={Annales scientifiques de l'Ecole normale sup{\'e}rieure},
  volume={39},
  number={3},
  pages={467--511},
  year={2006}
}

@article{GHKR1,
  title={Affine {Deligne--Lusztig} varieties in affine flag varieties},
  author={G{\"o}rtz, Ulrich and Haines, Thomas J. and Kottwitz, Robert E. and Reuman, Daniel C.},
  journal={Compositio Mathematica},
  volume={146},
  number={5},
  pages={1339--1382},
  year={2010},
  publisher={London Mathematical Society}
}

@article{HN,
 title={Minimal length elements of extended affine {Weyl} groups},
  author={He, Xuhua and Nie, Sian},
  journal={Compositio Mathematica},
  volume={150},
  number={11},
  pages={1903--1927},
  year={2014},
  publisher={London Mathematical Society}
}

@article{HHSZ,
  title={Parahoric reduction theory of formal connections (or {H}iggs fields)},
  author={Hu, Zhi and Huang, Pengfei and Sun, Ruiran and Zong, Runhong},
  journal={arXiv preprint arXiv:2409.05073},
  year={2024}
}

@book{HumphreysCoxeter,
  title={Reflection groups and {Coxeter} groups},
  author={Humphreys, James E.},
  volume={29},
  year={1990},
  publisher={Cambridge {U}niversity {P}ress, {C}ambridge}
}

@article{HV,
  title={The {N}ewton stratification on deformations of local {$G$}-shtukas},
  author={Hartl, Urs and Viehmann, Eva},
  journal={arXiv preprint arXiv:0810.0821},
  year={2008}
}

@incollection{Jantzen,
  title={Nilpotent orbits in representation theory},
  author={Jantzen, Jens Carsten},
  booktitle={Lie theory: Lie algebras and representations},
  pages={1--211},
  year={2004},
  publisher={Springer}
}

@article{Liu,
  title={$q$-Conjugacy classes in loop groups},
  author={Liu, Dongwen},
  journal={Proceedings of the American Mathematical Society},
  volume={140},
  number={9},
  pages={3297--3311},
  year={2012}
}

@article{Lus2,
  title={Elliptic elements in a {Weyl} group: a homogeneity property},
  author={Lusztig, George},
  journal={Representation Theory of the American Mathematical Society},
  volume={16},
  number={4},
  pages={127--151},
  year={2012}
}

@article{Lus1,
  title={From conjugacy classes in the {Weyl} group to unipotent classes},
  author={Lusztig, George},
  journal={Representation Theory of the American Mathematical Society},
  volume={15},
  number={14},
  pages={494--530},
  year={2011}
}

@article{NZ,
  title={On classification of $\sigma_q$-conjugacy classes of a loop group},
  author={Nie, Sian and Zhou, Peipei},
  journal={Journal of Algebra},
  volume={459},
  pages={29--42},
  year={2016},
  publisher={Elsevier}
}

@book{Ser,
  author    = {Serre, Jean-Pierre},
  title     = {Galois Cohomology},
  series    = {Springer Monographs in Mathematics},
  publisher = {Springer-Verlag},
  address   = {Berlin},
  year      = {2002},
  note      = {Translated from the French by Patrick Ion and revised by the author; corrected reprint of the 1997 English edition},
  isbn      = {978-3-540-42192-4},
  doi       = {10.1007/978-3-642-59141-9}
}

@article{St,
  title={Regular elements of semi-simple algebraic groups},
  author={Steinberg, Robert},
  journal={Publications Math{\'e}matiques de l'IH{\'E}S},
  volume={25},
  pages={49--80},
  year={1965}
}

@article{SS,
  author  = {Semenov-Tian-Shansky, M. A. and Sevostyanov, A. V.},
  title   = {Drinfeld--{S}okolov reduction for difference operators
             and deformations of {$W$}-algebras. {II}.
             {T}he general semisimple case},
  journal = {Communications in Mathematical Physics},
  volume  = {192},
  number  = {3},
  pages   = {631--647},
  year    = {1998},
  doi     = {10.1007/s002200050312}
}
\end{document}